\documentclass[12pt, a4paper]{article}
\usepackage[T1]{fontenc}
\usepackage[utf8]{inputenc}
\usepackage[english]{babel} 

\usepackage{fancyvrb}

\usepackage{mathtools}
\usepackage{amsmath}
\usepackage{amsthm}
\usepackage{amssymb}
\usepackage{stmaryrd}

\usepackage{ulem}

\usepackage{xparse}	

\usepackage{array}

\usepackage{algorithm}
\usepackage{algorithmic}
\usepackage{verbatim}
\usepackage[dvipsnames]{xcolor}

\usepackage{geometry}
\usepackage{makecell}

\usepackage{cleveref}
\usepackage{comment}
\usepackage{graphicx}
\usepackage{graphics}
\usepackage{subcaption}

\allowdisplaybreaks

\usepackage[backend=bibtex]{biblatex}
\usepackage{csquotes}
\newcommand{\N}{\mathbb N}

\newcommand{\R}{\mathbb R}
\newcommand{\C}{\mathbb C}

\newcommand{\abs}[1]{\left | #1 \right |}

\DeclareMathOperator{\sgn}{sgn}

\newcommand{\fctshortdef}[3]{ #1 : #2 \mapsto #3}

\newcommand{\fctlonganonymousdef}[4]{
	\left \{
	\begin{array}{rcl}
		#1	&	\longrightarrow	&	#2	\\
		#3	&	\longmapsto	&	#4	\\
	\end{array}
	\right.
}
\newcommand{\fctlongdef}[5]{
	#1 :
	\fctlonganonymousdef{#2}{#3}{#4}{#5}
}

\newcommand{\set}[2]{\left \{ #1 \, : \, #2 \right \}}
\newcommand{\quickset}[1]{\left \{ #1 \right \}}
\newcommand{\setindic}[1]{\mathbf{1}_{#1}}

\newcommand{\diff}[1]{\, d #1 \,}

\DeclareMathOperator*{\argmin}{argmin}

\newcommand{\eexp}{\operatorname e}

\DeclareMathOperator{\setspan}{span}

\newcommand{\vecthree}[3]{\begin{pmatrix} #1 \\ #2 \\ #3 \end{pmatrix}}
\DeclareMathOperator{\Tr}{Tr}

\newcommand{\norm}[1]{\left \| #1 \right \|}
\newcommand{\dotprodbracket}[2]{\langle #1, #2 \rangle}

\newcommand{\matset}[3]{{#1}^{#2 \times #3}}
\newcommand{\matsetsquare}[2]{\matset{#1}{#2}{#2}}

\newcommand{\dualitybracket}[2]{\langle #1, #2 \rangle}

\newcommand{\transpose}[1]{ #1^T }

\DeclareMathOperator{\supp}{supp}
\newcommand{\openball}[2]{B_{#2} \! \left( #1 \right)}

\newcommand{\Cregset}[1]{\mathcal C^{#1}}

\newcommand{\Cregsetsp}[2]{\Cregset {#1} \! \left( {#2} \right)}

\NewDocumentCommand{\smallO}{o m}{o_{\IfValueTF{#1}{{#1}}{}} \! \left( #2 \right)}
\NewDocumentCommand{\bigO}{o m}{\mathcal O \IfValueTF{#1}{_{#1}}{} \! \left( #2 \right)}

\newcommand{\setclosure}[1]{\overline{#1}}

\newcommand{\setweakclosure}[1]{ {\overline{#1}^w} }
\newcommand{\convexhull}[1]{\mathrm{conv} \left( {#1} \right)}

\newcommand{\laplacian}{\Delta}

\newcommand{\Lspace}[1]{L^{#1}}
\newcommand{\Lspacesp}[2]{L^{#1} \! \left( #2 \right)}
\newcommand{\Lone}{\Lspace 1}
\newcommand{\Lonesp}[1]{\Lspacesp 1 {#1}}
\newcommand{\Ltwo}{\Lspace 2}
\newcommand{\Ltwosp}[1]{\Lspacesp 2 {#1}}

\newcommand{\TestFunctionset}{\mathcal D}
\newcommand{\TestFunctionsetsp}[1]{\TestFunctionset \! \left( #1 \right)}
\newcommand{\Distribset}{\TestFunctionset'}
\newcommand{\Distribsetsp}[1]{\Distribset \! \left( #1 \right)}

\newcommand{\Hspacesp}[2]{H^{#1} \! \left( #2 \right)}

\newcommand{\Honesp}[1]{\Hspacesp 1 {#1}}

\DeclareMathOperator{\Fourier}{\mathcal F}

\newtheorem*{theorem*}{Theorem}

\newtheorem*{lemma*}{Lemma}
\newtheorem{lemma}{Lemma}
\newtheorem*{proposition*}{Proposition}
\newtheorem{proposition}{Proposition}
\newtheorem*{corollary*}{Corollary}
\newtheorem{corollary}{Corollary}
\newtheorem*{definition*}{Definition}
\newtheorem{definition}{Definition}
\newtheorem*{example*}{Example}

\newtheorem*{remark*}{Remark}
\newtheorem{remark}{Remark}

\numberwithin{theorem}{section}
\numberwithin{lemma}{section}
\numberwithin{proposition}{section}
\numberwithin{corollary}{section}
\numberwithin{definition}{section}
\numberwithin{example}{section}
\numberwithin{remark}{section}

\newcommand{\HHilb}{\mathcal H}

\title{Low-complexity approximations with least-squares formulation of the time-dependent Schrödinger equation}
\author{Mi-Song Dupuy, Virginie Ehrlacher, Clément Guillot}

\begin{document}

\maketitle

\begin{abstract}
	We propose new methods designed to numerically approximate the solution to the time dependent Schrödinger equation, based on two types of ansatz: tensors, and approximation by a linear combination of gaussian wave packets. In both cases, the method can be seen as a restricted optimization problem, which can be solved by adapting either the Alternating Least Square algorithm in the tensor case, or some greedy algorithm in the gaussian wavepacket case. We also discuss the efficiency of both approaches.
\end{abstract}

\section{Introduction} \label{sec:Intro}
Throughout this paper, we will discuss several low-complexity numerical methods for the resolution of the many-body time-dependent Schrödinger equation: for $T>0$ some final time, find $u\in L^2((0,T); \HHilb)$ such that
\begin{equation} \label{equ:Schrodinger equation}
	\begin{cases}
		i\partial_t u(t) = (H + B(t))u(t), \\
		u(0) = u_0,
	\end{cases}
\end{equation}
where $\HHilb$ a separable Hilbert space and $u_0\in \HHilb$ some initial condition. 
The operator $H$ is a (possibly unbounded) self-adjoint operator on $\HHilb$, and $(B(t))_{t\in (0,T)}$ a strongly continuous family of bounded self-adjoint operators on $\HHilb$.
This setup is motivated by quantum mechanics: for instance, for a molecular system with $N\in \mathbb{N}^*$ electrons, we will have $\HHilb = \Ltwosp{\R^{3N}}$, and
\begin{equation}\label{eq:electronic}
	H = - \sum_{i=1}^N \laplacian_{x_i} + \sum_{j=1}^N V_{ext}(t, x_j) + \sum_{1 \leq j < k \leq N} \frac{1}{\abs{x_j - x_k}},
\end{equation}
with $V_{ext}: \mathbb{R}^3 \to \mathbb{R}\cup\{\pm \infty\}$ an external potential, representing, for example, the interactions with the nuclei or an external electronic field. The practical resolution of this equation raises several difficulties. First, when $N$ is large, standard numerical approaches are prone to the curse of dimensionality and it thus becomes necessary to use low-complexity representations of the solution. Second, the equation must be solved in an unbounded domain (all of $\R^{3N}$), and due to the dispersive properties of the equation it is a priori impossible to localize the solution within a reasonably fixed large box. Although absorbing boundary conditions \cite{antoine_friendly_2017} can be used to recover the solution close to the nuclei, they are not capable of capturing the scattering part of the wave packet.
For all these reasons, it is often necessary to rely on a nonlinear approximation set, that is, a set which is not a linear subspace of the full space $\HHilb$, but only a subset $\Sigma \subset \HHilb$ of functions that can be represented with a low complexity \textit{i.e} a set with a small number of parameters).
Since this kind of approximation does not preserve the linear structure, it is no longer possible to rely on usual Galerkin discretization techniques combined with a time stepping scheme. The low-complexity subsets that have been used in practice are
\begin{enumerate}
    \item linear combinations of polynomial Gaussian wave packets~\cite{Hagedorn_1998, Lasser_Lubich_2020};
    \item structured low-rank tensor formats, such as \textit{Time-Dependent Density-Matrix Renormalization Group} (TD-DMRG, also called \textit{Tensor Trains})~\cite{Haegeman_Lubich_Oseledets_Vandereycken_Verstraete_2016} or hierarchical Tucker format under the name \textit{Multi-Configurational Time-Dependent Hartree} (MC-TDH)~\cite{Meyer_Manthe_Cederbaum_1990}.
\end{enumerate}

The most commonly employed approach to build a time-dependent low-complexity approximation $\widetilde{u}$ of $u$ out of a low-complexity subset $\Sigma$ is based on the so-called \textit{Dirac-Frenkel principle} \cite{falco2019dirac, faou2023modulation, faou2006poisson, feischl2024regularized, haegeman2016unifying, koch_dynamical_2007, lasser2022various, scheifinger2025time, secular2020parallel}.
The main principle is the following: assume that the approximation $\Sigma$ is a smooth manifold. Now, for each $x \in \Sigma$, denote by $T_x \Sigma$ the tangent space of $\Sigma$ at $x$, and $\pi_{T_x \Sigma}$ the orthogonal projector on this space. Then, the Dirac-Frenkel principle consists in solving the projected equation
\begin{equation} \label{equ:Dirac-Frenkel equation}
	\begin{cases}
		i\partial_t \tilde u(t) = \pi_{T_{\tilde u(t)}} ((H + B(t))\tilde{u}(t)), \\
		\tilde u(0) = \tilde u_0 \in \Sigma,
	\end{cases}
\end{equation}
where $\tilde u_0$ is an approximation of $u_0$ in $\Sigma$. If $\Sigma$ is a manifold of finite dimension, it can be (locally) parametrized, and \eqref{equ:Dirac-Frenkel equation} can be written as an ordinary differential equation in terms of the parameters. Equation~\eqref{equ:Dirac-Frenkel equation} is then solved using suitable time-stepping methods~\cite{hairer_geometric_2006}.

The Dirac-Frenkel principle has advantageous numerical properties. For example, the approximations that are obtained from the Dirac-Frenkel principle preserve the $\HHilb$ norm of the initial condition, as well as the energy (in the case $B(t)=0$). 
The main drawback of the Dirac-Frenkel principle is the lack of well-posedness for arbitrary times~\cite{koch_dynamical_2007}. 
In practice, it is often possible to circumvent this theoretical difficulty, but in this paper, our goal is to investigate another approach, based on a global space-time least squares formulation of the Schrödinger equation, which was studied in a previous work in~\cite{dupuy_space-time_2024}.

One of the advantage of using the global space-time approach is that, to prove the existence of a \itshape global-in-time \normalfont dynamical low-complexity approximation, one only needs the low-complexity subset $\Sigma$ to be a weakly closed subset of $\HHilb$. 
Moreover, the least-square formulation provides guaranteed and easy-to-compute a posteriori error estimators to assess the accuracy of the obtained dynamical low-complexity approximation.
In this paper, we discuss the advantages and disadvantages of this new approach when used in conjunction with low-rank tensor formats and linear combinations of Gaussian wave packets.

The outline of the present paper is the following: in Section~\ref{sec:Preliminaries}, we introduce some notation and recall some basic facts about the least-squares formulation of the time-dependent Schrödinger equation which were proved in~\cite{dupuy_space-time_2024}.  We also prove some interesting results related to general dynamical low-complexity approximations. In Section~\ref{sec:Tensors}, we focus our attention to the specific case where the low-complexity subset $\Sigma$ is given using some low-rank tensor format. We compare in particular on some toy numerical examples dynamical low-rank approximations obtained through the Dirac-Frenkel principle and the least-squares variational formulation and discuss about the pros and cons of both approaches. In Section~\ref{sec:Gaussians}, we comment on the use of the least-squares formulation in order to build a dynamical low-complexity approximation using the least-squares formulation. In particular, we show the interest of coupling this formulation together with some greedy algorithm in order to obtain some dynamical approximation of the solution of the time-dependent Schrödinger problem.



\section{Preliminaries}\label{sec:Preliminaries}

In this section, we recall some basic facts about the least-squares formulation of the time-dependent Schrödinger equation introduced in~\cite{dupuy_space-time_2024}. We also explain the general variational principle used in order to define low-complexity dynamical approximations and state some preliminary results.

\subsection{Least-squares formulation of the time-dependent Schrödinger equations}

In our previous work~\cite{dupuy_space-time_2024}, we considered the time-dependent equation (\ref{equ:Schrodinger equation}) with an Hamiltonian operator $H$ of the form $H(t) = H_0 + W + B(t)$ with $H_0$ any self-adjoint operator on $\HHilb$ with domain $D(H_0)$, $W$ a closed symmetric operator on $\HHilb$ such that $D(H_0) \subset D(W)$, and $B(t)$ a strongly continuous family of bounded self-adjoint operators.
In addition, we assume that there exists $\varepsilon>0$ with
$$
\mathop{\sup}_{\lambda \in \mathbb{R}} \|W (H_0 - \lambda \pm i \varepsilon)^{-1}\|_{\mathcal L(\HHilb)} < 1,
$$
where $\|\cdot\|_{\mathcal L(\HHilb)}$ denotes the operator norm on the set of bounded operators on $\HHilb$.

Let $u$ be the unique solution to (\ref{equ:Schrodinger equation}) and $I:=(0,T)$. Then, we can define $v \in \Honesp{I, \HHilb}$ by
\begin{equation}
\forall t\in [0,T], \quad v(t) = \eexp^{itH_0} u(t).
\end{equation}
The function $v \in \Honesp{I, \HHilb}$ is the only solution to
\begin{equation}
	\begin{cases}
		i\partial_t v = \eexp^{itH_0} (W + B(t)) \eexp^{-itH_0} v, \\
		v(0) = u_0. 
	\end{cases}
\end{equation}
The element $v$ can then be reformulated as the unique solution to the minimization problem
\begin{equation} \label{equ:Schrodinger variational}
	\text{Find $v \in \Honesp{I, \HHilb}$ such that} \quad v = \argmin_{w \in \Honesp{I, \HHilb}} F(w),
\end{equation}
with
\begin{equation} \label{equ:Variational functional}
	F(w)
	=	\norm{w(0) - u_0}_\HHilb^2 + T \norm{(i\partial_t - \eexp^{itH_0} (W + B(t)) \eexp^{-itH_0}) w}_{\Ltwosp{I, \HHilb}}^2.
\end{equation}
Indeed, it then holds that $F$ is a quadratic continuous strongly convex functional on $\Honesp{I, \HHilb}$.

\subsection{Dynamical low-complexity approximation}\label{sec:LS}

For a given low-complexity subset $\Sigma \subset \HHilb$, we consider the following variational principle to define some dynamical low-complexity approximation of $u$ using the least-squares formulation~\eqref{equ:Schrodinger variational}: 
\begin{equation}\label{equ:Schrodinger variationallowcomp}
	\text{Find $\widetilde{v} \in \Honesp{I, \Sigma}$ such that } \widetilde{v} \in \argmin_{\widetilde{w} \in \Honesp{I, \Sigma}} F(\widetilde{w}),
\end{equation}
where  $\Honesp{I, \Sigma} = \set{\widetilde{w} \in \Honesp{I, \HHilb}}{\widetilde{w}(t) \in \Sigma \text{ for any $t \in I$}}$.
An approximation of the initial time-dependent equation is retrieved from $\widetilde{u}$ defined by
\[
    \forall t\in [0,T], \quad \widetilde{u}(t):= \eexp^{-itH_0}\widetilde{v}(t).
\]

One advantage of formulation \eqref{equ:Schrodinger variationallowcomp} is the following result:
\begin{proposition} \label{prop: Weak closedness}
    Let $\Sigma$ be any non empty weakly closed subset of $\HHilb$. Then the subset $\Honesp{I, \Sigma}$ is closed in $\Honesp{I, \HHilb}$ for the weak topology.
\end{proposition}

\begin{proof}
    Since $\Honesp{I, \HHilb} \hookrightarrow \Cregsetsp{0}{I, \HHilb}$, for any $t \in I$ the linear application
    \[
        \fctlongdef{T_t}{\Honesp{I, \HHilb}}{\HHilb}{v}{v(t)}
    \]
    is continuous. \\
    Let $(v_n)_{n \geq 0} \subset \Honesp{I, \Sigma}$ be a sequence which weakly converges to some $v$ in $\Honesp{I, \HHilb}$. Since any bounded linear operator between Banach spaces is also weakly continuous, this implies that for any $t \in I$,
    \[
        v_n(t) = T_t v_n \mathop{\rightharpoonup}_{n\to +\infty} T_t v = v(t) \quad
        \text{in $\HHilb$},
    \]
    and the weak closedness of $\Sigma$ implies $v(t) \in \Sigma$.
\end{proof}

The following corollary is then an immediate consequence of Proposition~\ref{prop: Weak closedness} and of the fact that $F$ is a quadratic continuous strongly convex functional on $\Honesp{I, \HHilb}$. 
\begin{corollary}\label{cor:existence}
There always exists at least one solution $\widetilde{v}$ to~\eqref{equ:Schrodinger variationallowcomp}. 
\end{corollary}

\subsection{Importance of $H^1$ regularity}

We would like to emphasize here that the existence result stated in Corollary~\ref{cor:existence} only holds thanks to the sufficient regularity with respect to $t$. Indeed,  we also have the following result which states that $\Ltwosp{I, \Sigma}$ is not necessarily a weakly closed subset of $\Ltwosp{I, \HHilb}$, even if $\Sigma$ is a weakly closed subset of $\HHilb$.
\begin{proposition}
	Let $\Sigma \subset \HHilb$. Then
	\[
		\setweakclosure{\Ltwosp{I, \Sigma}}
		=	\Ltwosp{I, \setclosure{\convexhull \Sigma}}
	\]
	where $\setweakclosure{\Ltwosp{I, \Sigma}}$ is the closure of $L^2(I, \Sigma)$ in $L^2(I, \HHilb)$ for the weak topology, $\convexhull \Sigma$ is the convex hull of $\Sigma$, and $\setclosure{\convexhull \Sigma}$ is the closure of $\convexhull \Sigma$ in $\HHilb$.
\end{proposition}

\begin{proof}
	Let $C = \convexhull \Sigma$. It is immediate that $\Ltwosp{I, \setclosure C}$ is a convex and closed set for the norm topology of $L^2(I, \HHilb)$, hence weakly closed. Moreover, it contains $\Ltwosp{I, \Sigma}$, which gives the inclusion $\setweakclosure{\Ltwosp{I, \Sigma}} \subset \Ltwosp{I, \setclosure{\convexhull \Sigma}}$.
	
	It remains to prove the opposite inclusion.
	We first observe that the set of all $u \in \Ltwosp{I, \setclosure C}$ which can be written as $u=\sum_{i=1}^r \setindic{(a_i, b_i)} v_i$ with $(a_i, b_i)$ pairwise disjoint intervals and such that $\bigcup_i [a_i, b_i] = I$, and $v_i \in C$, is dense in $\Ltwosp{I, \setclosure C}$ for the norm topology, hence for the weak topology.
	Considering an element $u$ under the preceding form, it only remains to prove that there exists a sequence $(u_n)_{n\geq 0} \subset \Ltwosp{I, \Sigma}$ the weak limit of which is equal to $u$. For all $1\leq i\leq r$, there exist $P_i \in \mathbb{N}^*$ and $((\theta^i_p, \sigma^i_p))_{p=1,...,P_i} \in \left([0, 1] \times \Sigma\right)^{P_i}$ such that $\theta^i_1 + ... + \theta^i_{P_i} = 1$ and $\displaystyle v_i = \sum_{p=1}^{P_i} \theta^i_p \sigma^i_p$.
	For $s \in [0, 1)$, set
	\[
		f_i(s) = \sigma^i_p \quad \text{for $\theta^i_1 + ... + \theta^i_{p-1} \leq s < \theta^i_1 + ... + \theta^i_p$},
	\]
	which guarantees that $\int_0^1 f_i(s) \diff s = \sum_{p=1}^{P_i} \theta^i_p \sigma^i_p = v_i$, and extend $f_i$ to $\R$ by periodicity. Moreover, for almost all $s\in \mathbb{R}$, $f_i(s)\in \Sigma$ by construction. Set now, for all $\varepsilon>0$, 
	\[
	\forall t\in I, \quad	g_{i, \varepsilon}(t) = \setindic{(a_i, b_i)}(t) f_i \left( \frac{t}{\varepsilon} \right).
	\]
	It is a well-known homogenization result that $g_{i, \varepsilon} \xrightharpoonup[\varepsilon \to 0]{} \setindic{(a_i, b_i)} v_i$ in $L^2(I, \HHilb)$, and therefore that $\sum_{i=1}^r g_{i, \varepsilon} \xrightharpoonup[\varepsilon \to 0]{} u$ in $L^2(I, \HHilb)$. Moreover, since the intervals $(a_i, b_i)$ are pairwise disjoint, we indeed have $\sum_{i=1}^r g_{i, \varepsilon} \in \Ltwosp{I, \Sigma}$, which concludes the proof.
\end{proof}

The next sections of the article are devoted to the study of new methodologies to build dynamical low complexity approximations of the time-dependent Schrödinger equation, relying on the least-square variational formulation recalled above.

Section~\ref{sec:Tensors} is devoted to the case where $\Sigma$ is chosen as some low-rank tensor format whereas Section~\ref{sec:Gaussians} is focused on time-dependent gaussian wave packets approximations. \\
In both cases, the procedure can be roughly described as follows:
\begin{enumerate}
	\item Pick some weakly closed low-complexity set $\Sigma$ for the static problem.
	\item Use a suitable algorithm to find a solution to the global space-time problem
	\begin{equation}
		\argmin_{w \in \Honesp{I, \Sigma}} F(w).
	\end{equation}
\end{enumerate}

\section{Low-rank tensor formats} \label{sec:Tensors}

The aim of this section is devoted to some theoretical and numerical considerations about dynamical low-rank approximations of the time-dependent Schrödinger equation with the least-squares formulations presented in Section~\ref{sec:LS}.

\subsection{Finite-dimensional matrix case}

In practice, in the numerical test cases which will be highlighted below, we consider $\HHilb$ to be a finite-dimensional Hilbert space. More precisely, for the sake of simplicity, we consider some particular cases of \eqref{equ:Schrodinger equation} where $\HHilb = \mathbb{C}^{L_x \times L_y}$ for some $L_x, L_y\in \N^*$.

As a consequence, we consider various low-rank approximations of the solution to the matrix-valued Schrödinger equation
\begin{equation} \label{equ:Schrodinger matrix}
	\begin{cases}
		i\partial_t U(t) = \mathbb H(t, U(t)), \quad t\in I\\
		U(0) = U_0 \in \C^{L_x \times L_y},
	\end{cases}
\end{equation}
where $U: I \to \C^{L_x \times L_y}$ and  $\mathbb H: I \times \C^{L_x \times L_y} \to \C^{L_x \times L_y}$ is such that for each $t \in \R$, $\mathbb{H}(t, \cdot)$ is a self-adjoint linear operator with respect to the natural Frobenius inner product on matrices
\begin{equation}
	\forall M, N \in \C^{L_x \times L_y}, \quad
	\dotprodbracket{M}{N}_{\C^{L_x \times L_y}}
	=	\Tr{(M^*N)}.
\end{equation}
This is a particular case of  where $\HHilb = \C^{L_x \times L_y}$ is finite dimensional.
If $L_x$ and $L_y$ are large, a low-rank approximation can be used to significantly reduce the computational cost.
To this end, for a fixed rank $1 \leq r \leq \min(L_x, L_y)$, we define $\Sigma_r$ the set of all matrices of size $L_x \times L_y$ and rank lower than $r$
\begin{equation} \label{equ:Low rank set}
	\Sigma_r
	=	\set{M = AB \in \C^{L_x \times L_y}}{A \in \R^{L_x \times r}, \quad B \in \R^{r \times L_y}}.
\end{equation}

In that case, we would like to find an approximation $\widetilde{U} \in H^1(I,\Sigma_r)$ to the solution of the Schrödinger evolution~\eqref{equ:Schrodinger matrix}.
The set $\Sigma_r$ is not a smooth manifold. Nevertheless, it is still possible to compute dynamical low rank approximations using the so-called Dirac-Frenkel variational principle using for instance the algorithm proposed in~\cite{Ceruti_Kusch_Lubich_2022, lubich_projector-splitting_2014}. 
In this section, we want to compare the Dirac-Frenkel principle with our approach based on the global-space time formulation of the Schrödinger evolution.

We want to highlight that this simple example stands as a prototypical example of a discretization of high-dimensional problems using low-rank tensor formats. Possible low-rank tensors could be tensor trains~\cite{oseledets_tensor-train_2011} or hierarchical Tucker formats~\cite{Grasedyck_2010, Oseledets_Tyrtyshnikov_2009}. 
For simplicity, and since we want to compare the global-space time approach with the Dirac-Frenkel principle, we restrict the discussion to the matrix case.

\subsection{Description of the method (ALS)} \label{subsec:ALS tensor}
We describe here the algorithm that we use for our space-time minimization process in the context of a low-rank approximation to the solution of the Schrödinger equation.

Fix $r \in \mathbb{N}^*$ an approximation rank and $N \in \mathbb{N}^*$ the number of time steps, and consider the set
\begin{equation}
	V_{r, N}
	= \set{\sum_{k=0}^N \zeta_k(t) A_k \transpose{B_k}}{A_k \in \C^{L_x \times r}, \quad B_k \in \C^{L_y \times r}},
\end{equation}
where the $(\zeta_k)_{k=0,...,N}$ are the $\mathbb{P}^1$ hat functions associated to the uniform discretization mesh in time of the interval $(0,T)$ with time step $\Delta t = \frac T N$.
We also replace the functional $F$ defined in \eqref{equ:Variational functional} by a discrete version defined on $V_{r, N}$ by
\begin{equation}
	F_N \left( \sum_{k=0}^N \zeta_k A_k \transpose{B_k} \right)
	=	\int_0^T \diff t \norm{\sum_{k=0}^K i\zeta_k'(t) A_k \transpose{B_k} - \sum_{k=0}^N \zeta_k(t) \mathbb H \left( \frac{kT}{N}, A_k \transpose{B_k} \right)}^2, 
\end{equation}
where $\|\cdot\|$ denotes the Frobenius norm on $\C^{L_x\times L_y}$.
The difference between $F$ and $F_N$ is that the term $g(t) = H(t, \sum_{k=0}^N\zeta_k(t) A_k \transpose{B_k})$ is replaced by a discrete version $g_N(t) = \sum_{k=0}^N \zeta_k(t) \mathbb H \left( \frac{kT}{N}, A_k \transpose{B_k} \right)$ which allows direct computations of the integrals.
If we assume that $\mathbb H$ is smooth enough, we have
\begin{equation}
    \norm{g - g_N}_{\Ltwosp{I, \C^{L_x \times L_y}}}
    \leq    C \Delta t \norm g_{\Honesp{I, \C^{L_x \times L_y}}}.
\end{equation}
Therefore, it is reasonable to consider that, by taking a suficiently small $\Delta t$, the error introduced by the time discretization is  negligeable compared to the error induced by the truncation of the rank.
Therefore, in the numerical section, we will no longer consider this problem.

The problem to solve then reads as
\begin{equation} \label{equ:Schrodinger variational tensor}
	\text{Find $v \in V_{r, N}$ such that} \quad F_N(v) = \min_{w \in V_{r, N}} F_N(w).
\end{equation}
One may notice that $V_{r, N}$ is not \textit{stricto sensu} a subset of $\Honesp{I, \Sigma_r}$. Since enforcing this constraint would make the implementation significantly more challenging without any clear benefit, we choose to disregard this fact and work with $V_{r, N}$.

The problem \eqref{equ:Schrodinger variational tensor} is not linear, but it has a \textit{quadratic} structure with respect to the $A_k$'s and $B_k$'s. 
To reduce the complexity of the problem, we propose to solve it using an Alternating Least Squares (ALS) algorithm. 
In this algorithm, the minimization problem is solved by iteratively freezing the matrices $(A_k)$ and optimizing over $(B_k)$, then freezing the optimized $(B_k)$ and optimizing over the $(A_k)$. 
This process is repeated until convergence \cite{schollwock_density-matrix_2011, uschmajew_local_2012}.
\begin{algorithm}[h!]
	\caption{Alternating Least Square} \label{alg:ALS}
	\begin{algorithmic}
		\STATE \%Initialization can be random, or may come from an initial guess
		\STATE Initialize $(A_k)_{k=0,...,N}, (B_k)_{k=0,...,N}$.
		\WHILE{Number of iterations not reached}
		\STATE Compute
		\[
			(\tilde A_k)_{k=0,...,N} \in
			\argmin_{(C_k)_{k=0,...,N}} F_N \left( \sum_{k=0}^N \zeta_k(t) C_k \transpose{B_k} \right)
		\]
		\STATE Compute
		\[
			(\tilde B_k)_{k=0,...,N} \in
			\argmin_{(D_k)_{k=0,...,N}} F_N \left( \sum_{k=0}^N \zeta_k(t) \tilde A_k \transpose{D_k} \right)
		\]
		\STATE Update $(A_k)_{k=0,...,N} \gets (\tilde A_k)_{k=0,...,N}$
		\STATE Update $(B_k)_{k=0,...,N} \gets (\tilde B_k)_{k=0,...,N}$
		\ENDWHILE
	\end{algorithmic}
\end{algorithm}

Each step of the algorithm involves two linear problems of size $NL_xr$ for the $(A_k)$ and of size $NL_yr$ for the $(B_k)$.
Since both problems are very similar, we only focus on the linear system involving $(A_k)$.
The linear problems that we want to solve will be a block tridiagonal problem with blocks of size $L_xr$ for the $(A_k)$.
This means that we could solve each linear problem by performing a block factorization (LU or Cholesky), but the complexity would be $\bigO{N(rL_x)^3}$, which is costly if $L_x$ is too large.
Instead, we will use the conjugate gradient method to solve each linear problem.

This, however, requires a suitable preconditioner that we describe now. 
Let us define $M_0(B_0, ..., B_N)$ the matrix associated with the quadratic form
\begin{equation}
	F^0(A_0, ..., A_N)
	=	\norm{A_0}^2 + T \norm{\sum_{k=0}^N \zeta_k(t) A_k \transpose{B_k}}_{\Ltwosp{I, \C^{L_x \times L_y}}}^2
\end{equation}
then we can observe that $M_0$ has the form
\begin{equation}
	N_r \otimes I_{\C^{L_x \times L_x}},
\end{equation}
where $N_r$ is block tridiagonal with blocks of size $r \times r$, and only depends on the $B_k$'s.
Therefore, computing the inverse with a block factorization of $M_0$ only costs $\bigO{Nr^3}$, which is reasonable since $r$ is not meant to be large. Additionally, since $M_0$ only depends on the $B_k$'s, we only need to compute the block factorization once for each linear problem that we solve.
Naturally a symmetric precoditionner is used for the optimization with respect to the $B_k$'s. \\
The cost of the gradient descent is as follows: assume the cost of applying $H_x(t, \cdot)$ and $H_y(t, \cdot)$ to a vector is $\bigO{\gamma(\max(L_x, L_y))}$. For instance, if $H_x(t, \cdot)$ and $H_y(t, \cdot)$ are given as full matrices (resp. sparse matrices) we have $\gamma(n) = n^2$ (resp. $\gamma(n) = n$).
The cost of one iteration of the conjugate gradient descent is then $\bigO{NR^2\gamma(\max(L_x, L_y))}$.

\subsection{Numerical results} \label{subsec:Tensor numerical results}

In this section, we compare the numerical results obtained from the ALS algorithm to the most common Dirac-Frenkel principle.
For low-rank matrices, the Dirac-Frenkel method is presented in \cite{koch_dynamical_2007, lubich_projector-splitting_2014}. The precise implementation we used for the tests in this subsection is the second-order symmetric splitting described in \cite[Section 3.3]{lubich_projector-splitting_2014}. 
This integrator is very easy to implement and offers the significant advantage of being stable and robust under over approximation, as it does not involve any matrix inversion.

\subsubsection{A random matrix example}
Here we consider a small matrix example of the form \eqref{equ:Schrodinger matrix}, with $L_x = L_y = 40$ and
\begin{equation}
\begin{split}
	\mathbb H(t, M)
	&=	\chi(t) \eexp^{itH_{0,x}} H_{1,x} \eexp^{-itH_{0,x}} M \eexp^{itH_{0,y}} H_{1,y} \eexp^{-itH_{0,y}}	\\
	&\quad + (1 - \chi(t)) \eexp^{itH_{0,x}} H_{2,x} \eexp^{-itH_{0,x}} M \eexp^{itH_{0,y}} H_{2,y} \eexp^{-itH_{0,y}},
\end{split}
\end{equation}
where
\begin{itemize}
	\item $\chi(t) = \dfrac{1 + \cos(2\pi t)}{2}$;
	\item $H_{0, x}$ (resp. $H_{0, y}$) is a real diagonal matrix of size $L_x \times L_x$ whose coefficients are chosen according to a uniform law on $[0,1]$ (resp. $L_y \times L_y$);
	\item $H_{1, x}$ and $H_{2, x}$ (resp. $H_{1, y}$ and $H_{2, y}$) are real symmetric tridiagonal matrices of size $L_x \times L_x$ whose coefficients are chosen according to a uniform law on $[0,1]$. (resp. $L_y \times L_y$).
\end{itemize}
The initial condition is a rank~$1$ matrix of the form $X_0 \transpose{Y_0}$ with $X_0 \in \R^{L_x}$, $Y_0 \in \R^{L_y}$ two vectors with coefficient chosen randomly following a uniform law on $[0, 1]$.

We compare the accuracy of the space-time method, the Dirac-Frenkel principle, and the best possible approximation of rank $r$ computed from a reference solution (which we obtain through a classical RK4 method) using a truncated SVD.
We solve \eqref{equ:Schrodinger matrix} on the time interval $I  = (0,5)$, with $200$ time steps for both methods.

On \Cref{fig:sv tensor 1}, we plot the evolution of the singular values of the solution with respect to time. At $t=0$, only one singular value is nonzero since the initial condition is a rank~$1$ matrix, and more singular values begin to appear as time evolves.
On \Cref{fig:error tensor 1}, we compare the accuracy of the Dirac-Frenkel approach, and the space-time approach. We computed the error as $\sup_{t \in I} \norm{\tilde U(t) - U(t)}$, where $U$ is the reference solution, and $\tilde U$ is the approximated solution.
It turns out that both methods perform similarly in terms of accuracy and yield nearly optimal approximations compared to the best possible rank-$r$ approximation.

\begin{figure}[h]
	\centering
	\begin{subfigure}[t]{0.49\textwidth}
		\centering
		\includegraphics[width=\textwidth]{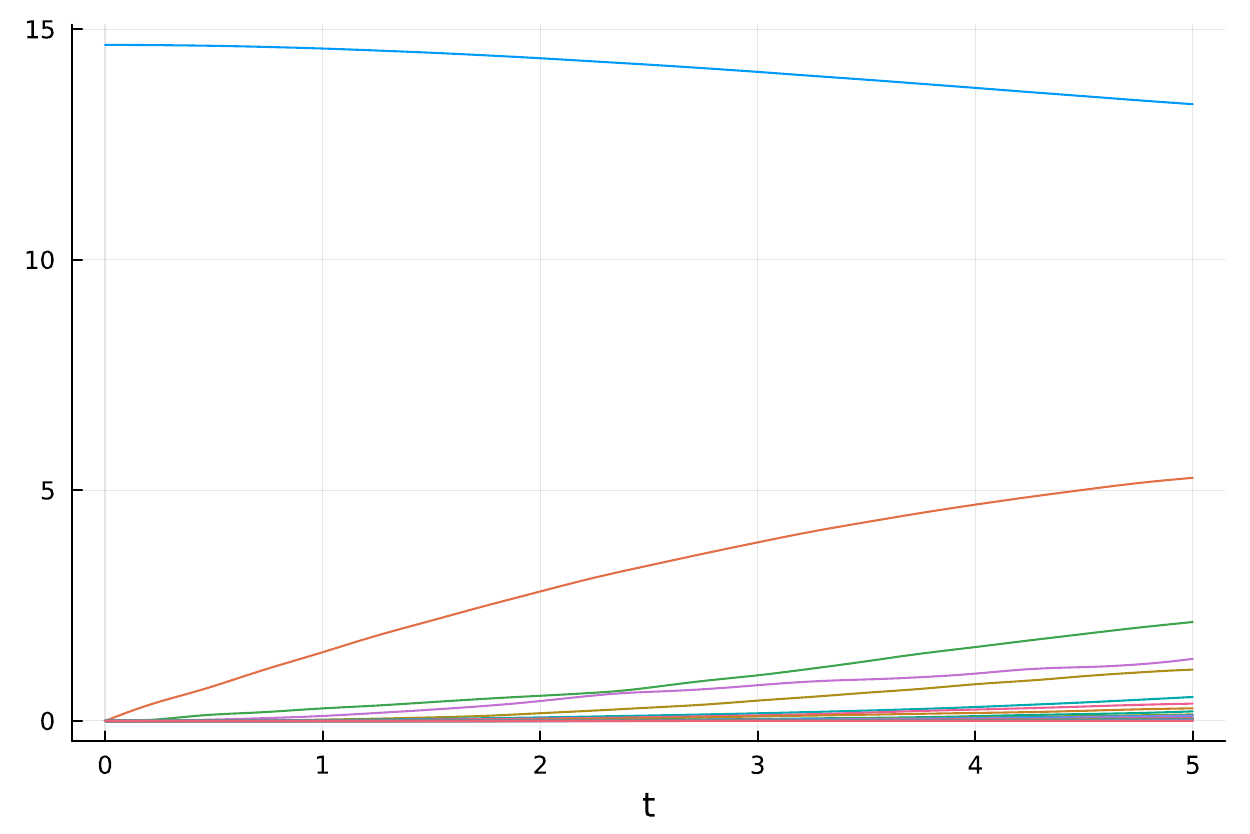}
		\caption{Evolution of the singular values of the solution.}
		\label{fig:sv tensor 1}
	\end{subfigure}
	\hfill
	\begin{subfigure}[t]{0.49\textwidth}
		\centering
		\includegraphics[width=\textwidth]{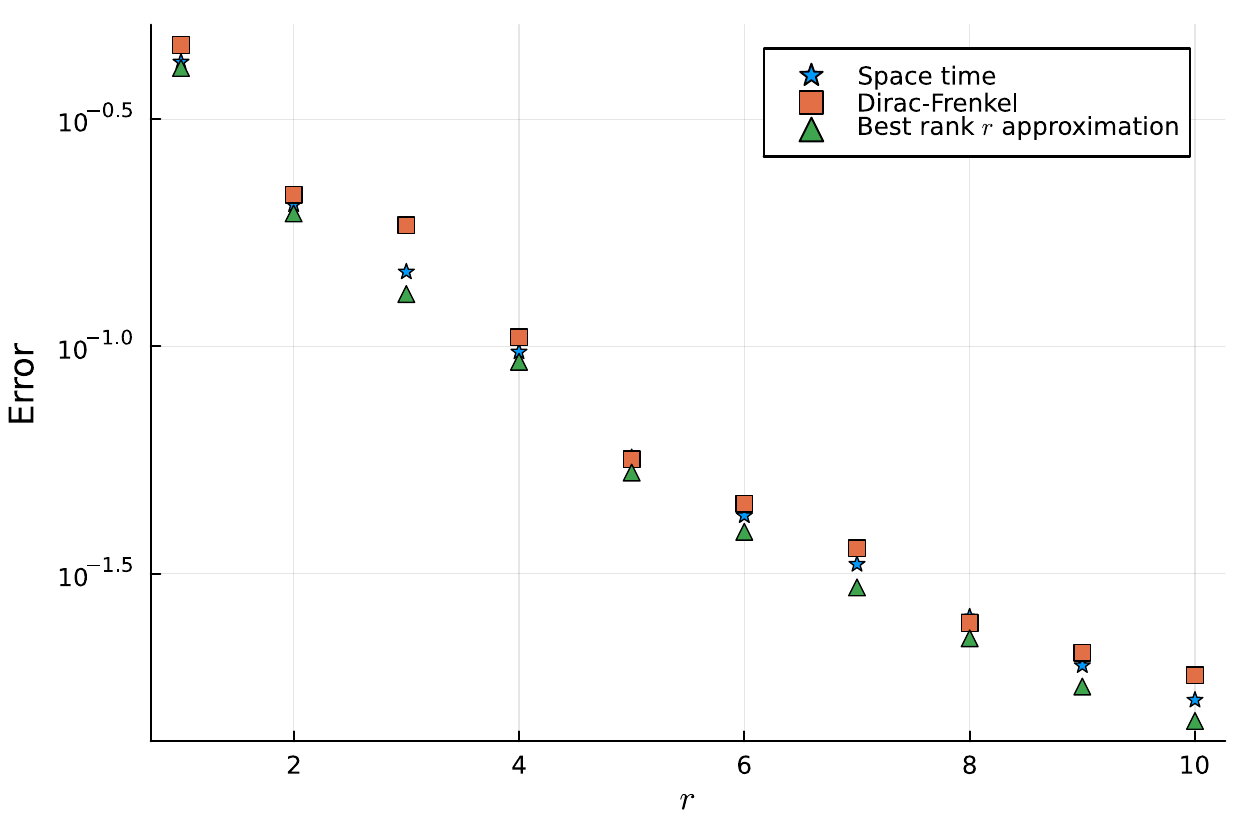}
		\caption{Error calculated as the maximum Frobenius norm of the difference between the computed solution and the reference solution.}
		\label{fig:error tensor 1}
	\end{subfigure}
	\caption{}
\end{figure}

\subsubsection{Pathological case}

We consider here a more challenging example for the Dirac-Frenkel principle.
It is designed such that the right hand side of the equation has a large orthogonal component relative to the tangent space of the approximation manifold.
This scenario presents difficulties for the Dirac-Frenkel principle in achieving an optimal, or near-optimal, approximation for a given fixed rank.
We choose here $L_x = L_y = L = 20$, $N = 200$, and consider the equation
\begin{equation} \label{equ:Pathological equation}
    \begin{cases}
        i\partial_t U(t) = H_x U(t) H_y, \\
        U(0) = U_0,
    \end{cases}
\end{equation}
where $U : I \to \C^{L \times L}$,
\(
    H_x = H_y =
    \begin{bmatrix}
    0 & I_{L/2} \\
    I_{L/2} & 0
    \end{bmatrix},
\)
where $I_n$ denotes the $n \times n$ identity matrix for any $n\in \N^*$, and
\[
    U_0 = {\rm diag}(1, \eexp^{-1}, \eexp^{-2}, ..., \eexp^{-L+1}).
\]
For each approximation rank $r$, we will replace the initial condition by its best approximation of rank $r$:
\[
    U_{0, r} = U_{0, x} \transpose{U_{0, y}},
\]
with
\[
    U_{0, x} =
    \begin{pmatrix} D_r \\ 0 \end{pmatrix}\in \C^{L \times r},
    \quad
    U_{0, y} =
    \begin{pmatrix} I_r \\ 0 \end{pmatrix} \in \C^{L \times r},
\]
where $D_r = {\rm diag}(1, \eexp^{-1}, ..., \eexp^{-r+1})$. \\
This particular design is motivated by the fact that the right hand side of \eqref{equ:Pathological equation} lies exactly within the orthogonal to the tangent space, meaning that the approximation computed by the Dirac-Frenkel principle will simply remain constant.
Since we believe that this comparison would be too unfair for the Dirac-Frenkel principle and is irrelevant in practice, some small random noise is added to the initial condition for the Dirac-Frenkel computations.
More precisely, we change the initial condition to $\tilde U_{0, r} = \tilde U_{0, x} \transpose{\tilde U_{0, y}}$ where
\[
    \tilde U_{0, x} =
    U_{0, x} + \delta U_{0,x} \in \C^{L \times r},
    \quad
    \tilde U_{0, y} =
    U_{0, y} + \delta U_{0,y} \in \C^{L \times r}.
\]
The coefficient of the noise matrices $\delta U_{0, x}$ and $\delta U_{0, y}$ are chosen according to a uniform law on $[0, \varepsilon]$. As shown in \Cref{fig:error tensor 2} with the case $\varepsilon=0$, the error does not decay with this random noise.

In Figure~\ref{fig:Pathological case}, the left panel displays the behavior of the singular values of the exact solution $U(t)$ to~\eqref{equ:Schrodinger matrix}. 
As expected from the design of the operators $H_x, H_y$, there are many singular values crossings in the evolution, rendering difficult the resolution using a Dirac-Frenkel principle.
Nevertheless, the Dirac-Frenkel principle gives results within an order of magnitude, even for very small values of $\varepsilon$.
Note that the approximation given by the global space-time approach is less sensitive than the Dirac-Frenkel principle, giving results close to the optimal solution.
However, this observation must be balanced by comparing the computational cost of the two methods. On \Cref{fig:time tensor 2}, we display the computation time of the two methods with respect to the error obtained. We find that, for a given error tolerance, the Dirac-Frenkel method (while potentially requiring a higher rank) ultimately proves to be more computationally efficient.

\begin{figure}[h]
	\centering
	\begin{subfigure}[t]{0.49\textwidth}
		\centering
		\includegraphics[width=\textwidth]{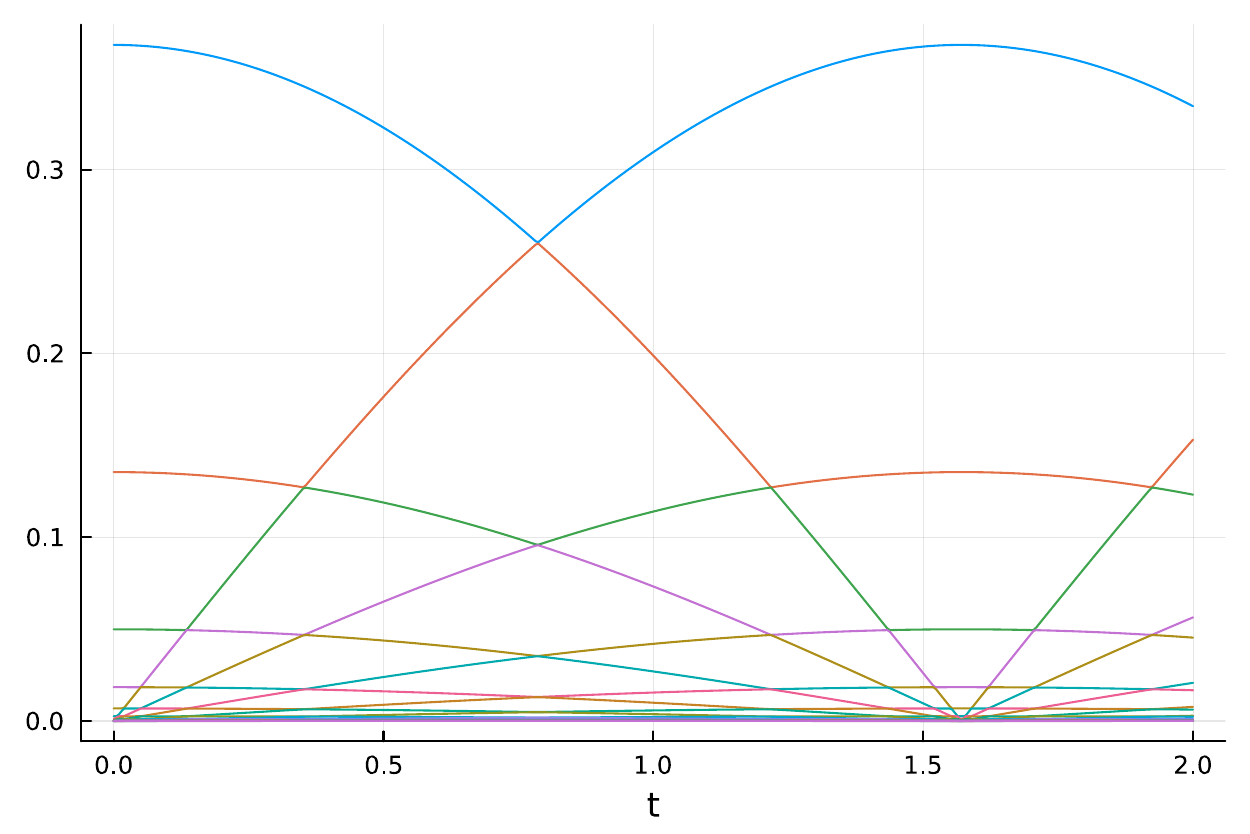}
		\caption{Evolution of the singular values of the solution with respect to time.}
		\label{fig:sv tensor 2}
	\end{subfigure}
	\hfill
	\begin{subfigure}[t]{0.49\textwidth}
		\centering
		\includegraphics[width=\textwidth]{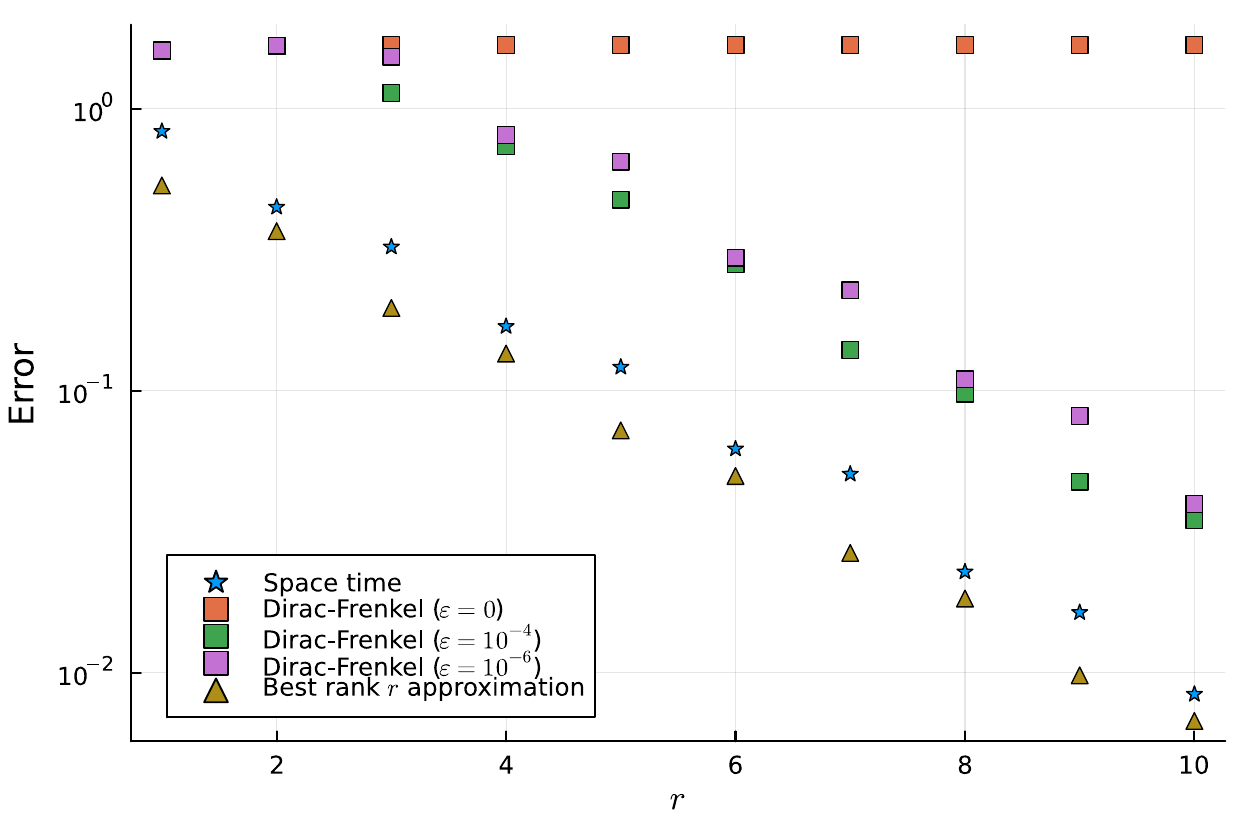}
		\caption{Error calculated as the maximum Frobenius norm of the difference between the computed solution and the reference solution at $t=2$ with respect to the approximation rank.}
		\label{fig:error tensor 2}
	\end{subfigure}

    \begin{subfigure}[t]{0.5\textwidth}
		\centering
		\includegraphics[width=\textwidth]{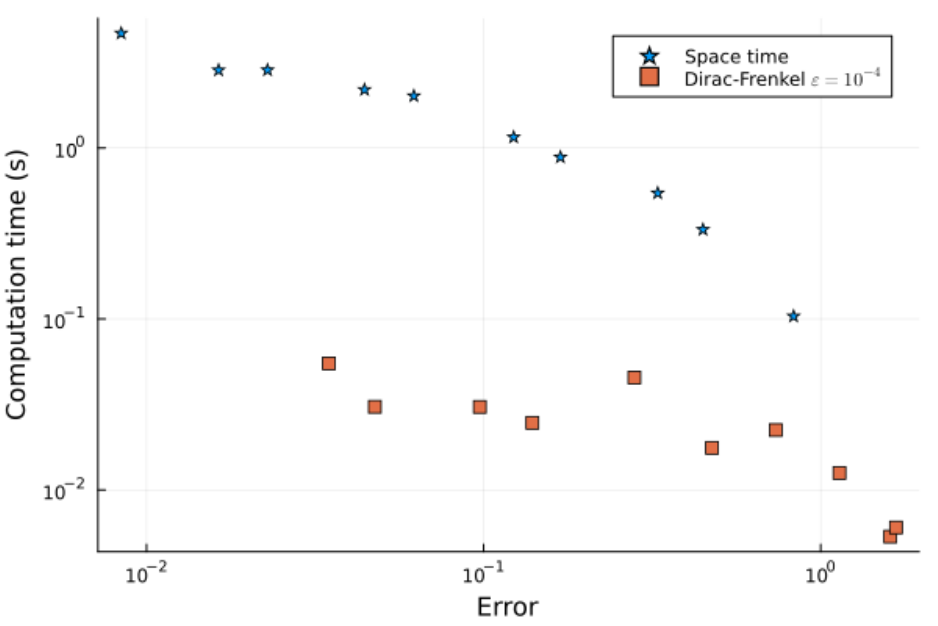}
		\caption{Computation times with respect to error for the different methods.}
		\label{fig:time tensor 2}
	\end{subfigure}
    
	\caption{} \label{fig:Pathological case}
\end{figure}

\section{Gaussian wave packets} \label{sec:Gaussians}
In this section, we introduce an algorithm aimed at approximating the solution of the time-dependent Schrödinger equation as a linear combination of Gaussian wave packets, utilizing a least-squares formulation.
First, let us introduce some notation. For a fixed value of $d \in \N^*$, let $\mathcal S_d$ denote the set of symmetric matrices of $\R^{d\times d}$ and $\mathcal S^{+,*}_d$ denote the set of positive-definite symmetric matrices of $\mathbb{R}^{d \times d}$.

Let us also introduce the following sets:
\begin{equation} \label{equ:gaussian ansatz}
	\mathcal G
	=	\set{g: \mathbb{R}^d \to \mathbb{C}}{ \begin{array}{c}
    \forall x\in \mathbb{R}^d, \; g(x) = a \eexp^{- \frac 1 2 (x-q) \cdot Q(x - q)} \eexp^{i  p \cdot (x-q)}, \\ \mbox{with } a \in \C,\; p, q \in \R^d, \; Q = A + iB \; \mbox{ with }A \in \mathcal S_d^{+,*} \mbox{ and } B \in \mathcal S_d\\
    \end{array}}.
\end{equation}
and, for a given collection of non-negative integer $\nu = (\nu_1, ..., \nu_d) \in \N^d$:

\begin{equation}
	\mathcal G_\nu
	=	\quickset{P g \ : \ P \in \C^{\nu}[x_1, ..., x_d],\;  g \in \mathcal{G}},
\end{equation}
where $\C^\nu[x_1, ..., x_d] = \set{\sum_{\theta_d \leq \nu_1, ..., \theta_d \leq \nu_d} a_\alpha x_1^{\theta_d} ... x_d^{\theta_d}}{(a_\alpha)_\alpha \subset \C}$.
Notice then that $\mathcal G_{(0,\ldots,0)} = \mathcal G$. 


This format is classically used to compute approximations of solutions to the Schrödinger equation \cite{Sawada_Heather_Jackson_Metiu_1985,joubert-doriol_variational_2018, kvaal_no_2023,burkhard2024variational}.
It offers several advantageous features including high flexibility, and a natural and relatively compact parametrization. The gaussian wave packets also naturally appear in the context of semiclassical behavior \cite{lasser_computing_2020}. In practice, the solution is typically sought as a linear combination of such (polynomial) gaussian wave packets, and the parameters are updated following a time-stepping procedure.
The most fundamental approach relies on the Dirac-Frenkel principle with various splitting methods \cite{hairer_geometric_2006}. However, it often happens that the gaussian wave packets begin to overlap each other as time moves on, which results in a very ill-conditioned Gram matrix and ultimately leads to some numerical instability.
Alternative techniques have been developed to avoid this phenomenon \cite{joubert-doriol_variational_2018, kvaal_no_2023}.
In the following sections, we propose another approach based on the global space-time formulation which has the advantage of not suffering from the limitations highlighted above.


The outline of the section is the following:
In \Cref{subsec:Gaussian minimization method} we describe a greedy procedure combining the ansatz $\mathcal G$ with the least squares formulation \eqref{equ:Schrodinger variational}. 
In particular, we discuss about the preconditioned gradient descent needed to solve the minimization problem with gaussian wave packets. 
In Section \ref{subsec:Gaussian numerical result}, we numerically test this procedure on two toy problems: first, a one-dimensional problem, followed by a three-dimensional example.

\subsection{Description of the greedy method} \label{subsec:Gaussian minimization method}

\subsubsection{The discrete minimization problem}

Let us describe the process through which we hope to obtain a good approximation of the solution.
Here, we shall only consider the electronic Schrödinger equation
\begin{equation}
	\begin{cases}
		i\partial_t u(t, x) = - \laplacian u(t, x) + V(t, x) u(t, x), \\
		u(0, x) = u_0(x),
	\end{cases}
\end{equation}
so that we are searching for $v$ solution to the minimization problem
\begin{equation}
	v = \argmin_{w \in \Honesp{I, \Ltwosp{\R^d}}} F(w),
\end{equation}
with
\begin{equation}
	F(w)
	=	\norm{w(0) - u_0}_{\Ltwosp{\R^d}}^2 + T \norm{(i\partial_t - \eexp^{-it\laplacian} V \eexp^{it\laplacian}) w}_{\Ltwosp{I \times \R^d}}^2.
\end{equation}

First, we need to discretize $\Honesp{I, \mathcal G}$. We only consider the case $d=1$, the general case is easily deduced. For $N \geq 2$, let $(\zeta_k)_{k=0,...,N}$ be the uniform P1 hat functions on $[0, T]$, that is, such that
\begin{equation}
	\zeta_k \! \left(l \frac T N \right) = \delta_{kl}.
\end{equation}
We then define the set
\begin{equation}
	W_N
	=	\set{\sum_{k=0}^N \zeta_k(t) g_k(x)}{(g_k)_{k=0,...,N} \subset \mathcal G},
\end{equation}
It is worth noting that, similarly to the tensor case, $W_N$ is not a subset of $\Honesp{I, \mathcal G}$. Since there is no practical reason to enforce this constraint, we choose to work with $W_N$ nevertheless because it makes the implementation easier. \\
In the discrete version of the problem, for a given element $w \in W_N$, we will also replace $\eexp^{-it\laplacian} V \eexp^{it\laplacian}w$ by its discrete version acting on $W_N$:
\[
	\sum_{k=0}^N \zeta_k(t) \eexp^{-i \frac{kT}{N} \laplacian} V \eexp^{i \frac{kT}{N} t\laplacian} g_k.
\]
Unless we are dealing with the trivial case where $V$ is constant, there is no reason why $v$ should be well approximated by a single element of $W_N$. One could try to search for the solution as a sum of elements of $W_N$, and optimize all the parameters at once, but the elements of the sum could potentially overlap which would lead to an ill-conditioned problem.
Moreover, although the block structure described below would still be present, it would involve larger blocks, leading to a much more expensive computation.
For all these reasons, we choose to compute the terms one by one, iteratively, using a greedy algorithm such as \Cref{alg:gaussian greedy algo}.
\begin{algorithm}[h!]
	\caption{Greedy Algorithm} \label{alg:gaussian greedy algo}
	\begin{algorithmic}
		\STATE Initialize \( r \gets 1 \).
		\STATE \%The stopping criteria can be defined based on a condition on the residual or by computing a fixed number of terms.
		\WHILE{Stopping criteria not satisfied}
		\STATE Compute \( G_r \gets \argmin_{G \in W_N} F(\sum_{j=1}^{r-1} G_j + G) \).
		\STATE Update \( r \gets r+1 \).
		\ENDWHILE
	\end{algorithmic}
\end{algorithm}

\subsubsection{Convergence of the greedy algorithm}

The convergence of the greedy algorithm requires to introduce the notion of dictionaries.
\begin{definition}
	Let $E$ be a complex Banach space. A subset $\Sigma \subset E$ is called a \textit{dictionary} if it satisfies the following properties:
	\begin{itemize}
		\item[i)]
		$
			\forall x \in \Sigma, \quad \forall \lambda \in \C, \quad
			\lambda x \in \Sigma
		$
		\item[ii)] $\Sigma$ is weakly closed in $E$
		\item[iii)] $\setspan \Sigma$ is dense in $E$.
	\end{itemize}
\end{definition}

Greedy algorithms have been studied for different framework in \cite{cances_convergence_2011,falco2012proper, devore_remarks_1996,ammar_convergence_2010}. The main result of interest to us is the following, the proof can be found in \cite{ammar_convergence_2010} in the case of rank~$1$ tensors but can be directly adapted to more general dictionaries:
\begin{proposition} \label{prop:Greedy algorithm convergence}
    Let $F$ be a quadratic strongly convex functional, and $\Sigma$ be a dictionary.
    Then the sequence $\left( \sum_{j=1}^r G_j \right)_{r \in \N}$ defined by \Cref{alg:gaussian greedy algo} converges in $\Honesp{I, \Ltwo}$ to $\argmin_{w \in \Honesp{I, \Ltwo}} F(w)$.
\end{proposition}

The fact that \Cref{prop:Greedy algorithm convergence} can be applied here is simply a consequence of the following result:
\begin{proposition} \label{prop:Gaussian dictionary}
    Let $d \geq 1$.
    For any $\nu = (\nu_1, ..., \nu_d) \in \N^d$, $\mathcal G_\nu$ and $1 \leq r < \infty$, $\mathcal G_\nu$ is a dictionary in $\Lspacesp{r}{\R^d}$.
    Consequently, $\Honesp{I, \mathcal G_\nu}$ is also a dictionary in $\Honesp{I, \Ltwosp{\R^d}}$.
\end{proposition}

\begin{proof}
    The fact that $\mathcal G_\nu$ is a dictionary in $\Lspace r$ is proved in \Cref{app:Proof of weak closedness}.
    The second part of the proposition follows from \Cref{prop: Weak closedness}.
\end{proof}

\begin{remark} \label{rem:Sum of gaussians not dictionary}
	It is informative to notice that the set $\mathcal G + \mathcal G$ of functions which read as a linear combination of gaussian wave packets lacks the property of weak closedness and thus does not qualify as a dictionary.
	For instance, it can be checked that for any $1 \leq r < \infty$,
	\[
		\frac{g_h -g}{h} \xrightharpoonup[h \to 0]{\Lspace r} w,
	\]
    where for all $x\in \mathbb{R}^d$, $g_h(x) = \eexp^{-(x+h)^2}$, $g(x) = \eexp^{-x^2}$ and $w(x) = -2x \eexp^{-x^2}$ so that $w\notin \mathcal G$. Similary, the set of functions which read as a linear combination of $K$ gaussian wave packets is in general not closed as soon as $K\geq 2$.
\end{remark}

\subsubsection{Description of one greedy step}

We describe here the minimization process involved for one step of the greedy algorithm.
First, let us fix a proper parametrization for our problem.
There is a natural parametrization
\begin{equation}
	\fctlongdef{\gamma}{\R^6}{\Ltwosp{\R}}{(v_0, ..., v_5)}{(v_0+iv_1)\eexp^{-\frac{(v_2+iv_3)}{2} (x - v_4)} \eexp^{iv_5x}},
\end{equation}
which maps a vector $\boldsymbol v = (v_0, ..., v_5) \in \R^6$ to an element of $\mathcal G$ (again, here, we fixed $d=1$).
Then one can then parametrize $W_N$ as follows:
\begin{equation}
    \fctshortdef{\Gamma}{X = (x_0, ..., x_{6(N+1)-1}) \in \R^{6(N+1)}}{\sum_{k=0}^N \zeta_k(t) \gamma(X_{6k:6k+5})} \in W_N,
\end{equation}
where for $i \leq j$ we define $X_{i:j} = (x_i, x_{i+1}, ..., x_{j-1}, x_j)$.
The discretization of \eqref{equ:Schrodinger variational} gives
\begin{equation} \label{equ: Discrete Variational formulation}
    \text{Find } X_* \in
    \argmin_{X \in \R^{6(N+1)}} F(X),
\end{equation}
where
\begin{equation}
    F(X)
    =   \norm{g_0 - u_0}_{\Ltwosp{\R}}^2 + T \norm{\sum_{k=0}^N \left( i\partial_t - H\!\left( l \frac T N \right) \right) \zeta_k(t) \gamma(X_{6k:6k+5}) - f}_{\Ltwosp{I \times \R}}^2.
\end{equation}
If $f(t)$ and $u_0$ can be written as a sum of polynomial-gaussians, $F(X)$ has an analytic expression, and its gradient can be computed analytically as well (in \cite{guillot_clguillotschrodingergaussian_2025} we compute the gradient with Julia's ForwardDiff package \cite{RevelsLubinPapamarkou2016}). It is therefore tempting to use a gradient descent to solve the discrete minimization problem. However, the convergence turns out to be remarkably slow, due to the ill-conditioning of the parametrization, that is, the natural Euclidean metric on $\R^{6(N+1)}$ fails to reproduce the metric associated with $F$ on $W_N$.
In principle, this problem could be solved by replacing the gradient descent by a Newton method. Since each $\zeta_k$ only "interacts" with $\zeta_{k-1}$ and $\zeta_{k+1}$, the hessian of $F$ with respect to $X$ is in fact block-tridiagonal with blocks of size $6 \times 6$. This makes the Newton method computationally affordable. However, computing the blocks of the hessian of $F$, even by taking advantage of the sparsity, turns out to be quite expensive (much more than the gradient). Moreover, as the greedy algorithm moves on, each previous term is added to the right member $f$, and the cost of computing the hessian increases even more.

A possible way to fix this issue is to rely on an idea similar to the "natural gradient descent" method \cite{schollwock_density-matrix_2011} which has gained a lot of popularity in the field of scientific machine learning recently. We define the function
\begin{equation}
    G(X_1, X_2)
    =   \dotprodbracket{\Gamma(X_1)(0)}{\Gamma(X_2)(0)}_{\Ltwosp \R} + T \dotprodbracket{i\partial_t \Gamma(X_1)}{i\partial_t \Gamma(X_2)}_{\Ltwo_t \Ltwosp \R}.
\end{equation}
For a given $X \in \R^{6N}$, we compute the matrix
\[
	\tilde H(X) = \partial_{X_1} \partial_{X_2} G(X, X).
\]
This matrix represents a metric equivalent to the functional $F$ projected on the tangent space of $\mathcal V$ parametrized by $\R^{6N}$, and can be used as a preconditioner for the gradient descent.
There are several benefits compared to the Newton method:
\begin{itemize}
	\item $\tilde H(X)$ is guaranteed to be nonnegative.
	\item Computing $\tilde H(X)$ is cheaper than computing the hessian of $F$. Indeed, it shares the same property of being block tridiagonal, but the underlying function is much more simple (in particular, the cost is the same for every step of the greedy algorithm).
\end{itemize}
Now we use \Cref{alg:one gaussian opti} for the one-gaussian optimization.

\begin{algorithm}[h!]
	\caption{Optimization Algorithm} \label{alg:one gaussian opti}
	\begin{algorithmic}
		\STATE Initialize \( X \).
		\STATE Set \( \varepsilon \gets \infty \).
		\WHILE{\( \varepsilon > \varepsilon_{\text{lim}} \)}
		\STATE Compute \( Y \gets \tilde{H}(X)^{-1} \nabla F(X) \).
		\STATE Find $\alpha_{\mathrm{opt}}$ such that  \(F(X - \alpha_{\mathrm{opt}} Y) = \min_{\alpha \in \R} F(X-\alpha Y) \).
		\STATE Update \( X \gets X - \alpha_{\mathrm{opt}} Y \).
        \STATE Update $\varepsilon = Y \cdot \nabla F(X)$
		\ENDWHILE
	\end{algorithmic}
\end{algorithm}
Taking advantage of sparsity, the cost of computing the gradient is $\bigO{N n_f}$ where $n_f$ is the number of terms in $f(t)$.
Computing $\tilde H(X)$ costs around $\bigO{N}$ operations, and the inverse is computed with a block Cholesky factorization ($\tilde H(X)$ is block-tridiagonal) which costs about the same.
One has to keep in mind that the constants of the $\mathcal O$ also depend on the number of parameters used for the wave packets. Specifically, as this number increases—such as in higher-dimensional problems—it may become advantageous to employ a different method for computing derivatives (e.g., reverse differentiation) in order to reduce the asymptotic cost.

\subsection{Numerical results} \label{subsec:Gaussian numerical result}

The code of the results in this section can be found in \cite{guillot_clguillothermitewavepackets_2025, guillot_clguillotschrodingergaussian_2025}.

\subsubsection{Reflection of a wave packet on a potential barrier} \label{subsec:Numerical 1D wave packet}
In this numerical example, we study the scattering of a gaussian wave packet
\begin{equation}
	g_0(x)
	=	\eexp^{-\frac 1 2 (x - q)^2} \eexp^{ipx},
\end{equation}
which corresponds ``semi-classically'' to a particle initially located in $q$ with a velocity $p$, moving toward a ``double hump'' potential barrier
\begin{equation}
	V(x)
	=	1.5 \eexp^{-\frac{(x+2)^2}{2}} + \eexp^{-\frac{(x-2)^2}{2}}.
\end{equation}
We solve the modified dynamics
\[
\begin{cases}
	i\partial_t \varphi = \eexp^{-it\laplacian} V \eexp^{it\laplacian} \varphi, \\
	\varphi(0) = g_0.
\end{cases}
\]
Of course, we easily recover the real solution $\psi(t) = \eexp^{it\laplacian} \varphi(t)$.

We apply the procedure described in \Cref{subsec:Gaussian minimization method}, with the following parameters:
\begin{equation}
	q=6, \quad p=-1, \quad T=5, \quad N=100.
\end{equation}
This choice of parameters is motivated by the observation of partial tunneling effects and ensures that the time frame captures a significant part of the interaction between the wave packet and the barrier. The wave packet is expected to undergo partial reflection and partial transmission upon encountering the barrier.

On \Cref{fig:wave evolution}, we plot the density $\abs \psi^2$ associated to the solution at different times $t$. On \Cref{fig:norm evolution}, we display and the evolution of the $\Ltwo$ norm of the computed solution. As expected, the norm is not preserved but is close to the exact norm.
On \Cref{fig:residual evolution}, we plot the decay of the residual $F$ with respect to the number of terms. We observe a slow but steady decay, typical for greedy algorithms.
\begin{figure}[h]
	\centering
	\begin{subfigure}[t]{0.49\textwidth}
		\centering
		\includegraphics[width=\textwidth]{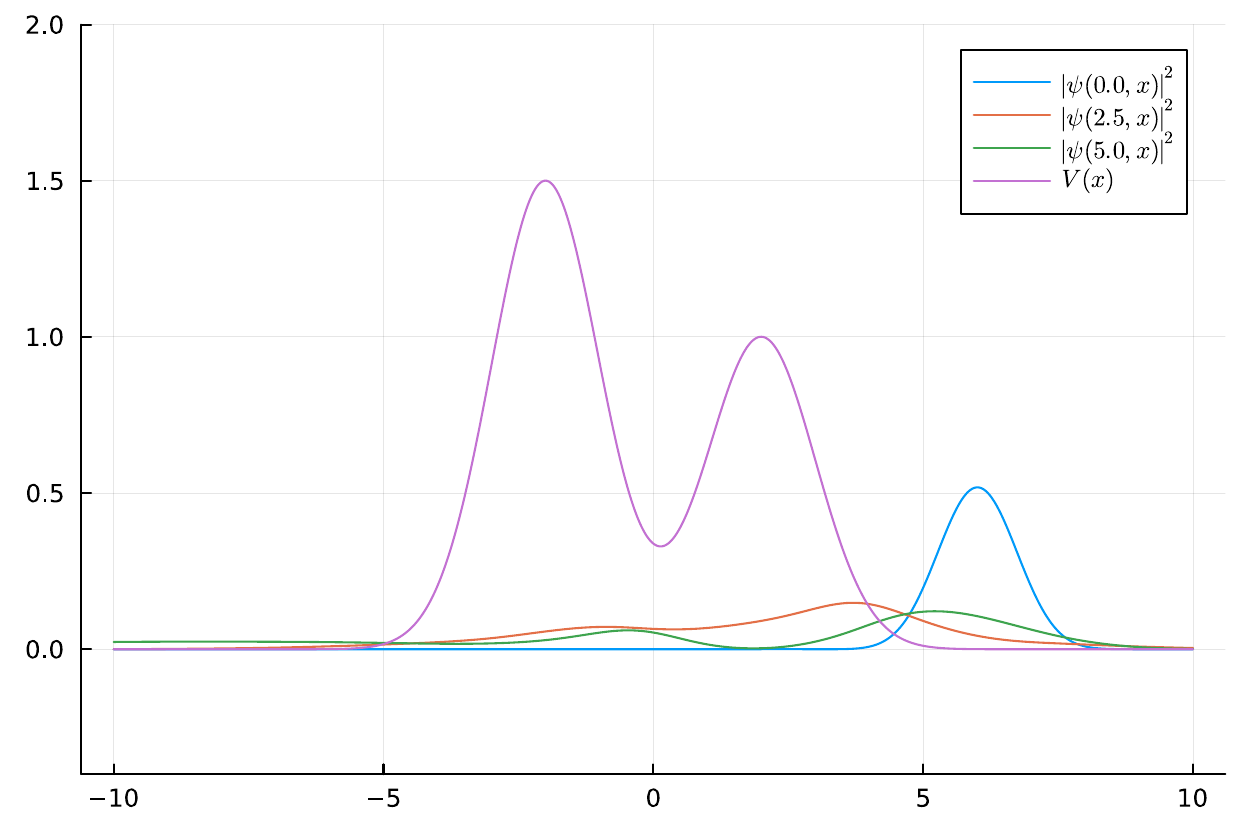}
		\caption{Representation of the square modulus of the wave function $\psi$ for different values of $t$.}
		\label{fig:wave evolution}
	\end{subfigure}
	\hfill
	\begin{subfigure}[t]{0.49\textwidth}
		\centering
		\includegraphics[width=\textwidth]{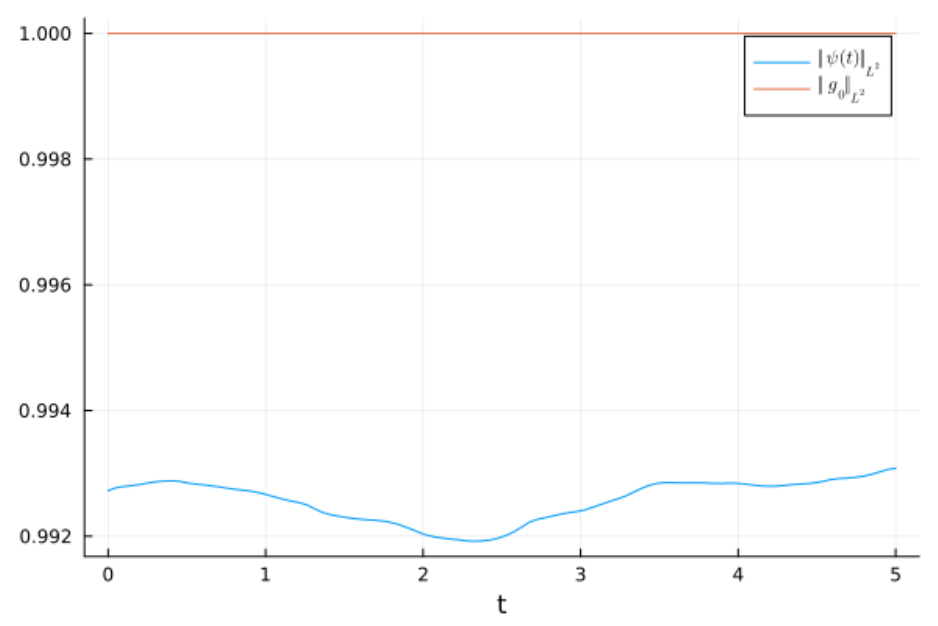}
		\caption{Evolution of the $L^2$ norm of the solution computed by the orthogonal greedy algorithm with $30$ terms.}
		\label{fig:norm evolution}
	\end{subfigure}	
	\begin{subfigure}[t]{0.5\textwidth}
		\centering
		\includegraphics[width=\textwidth]{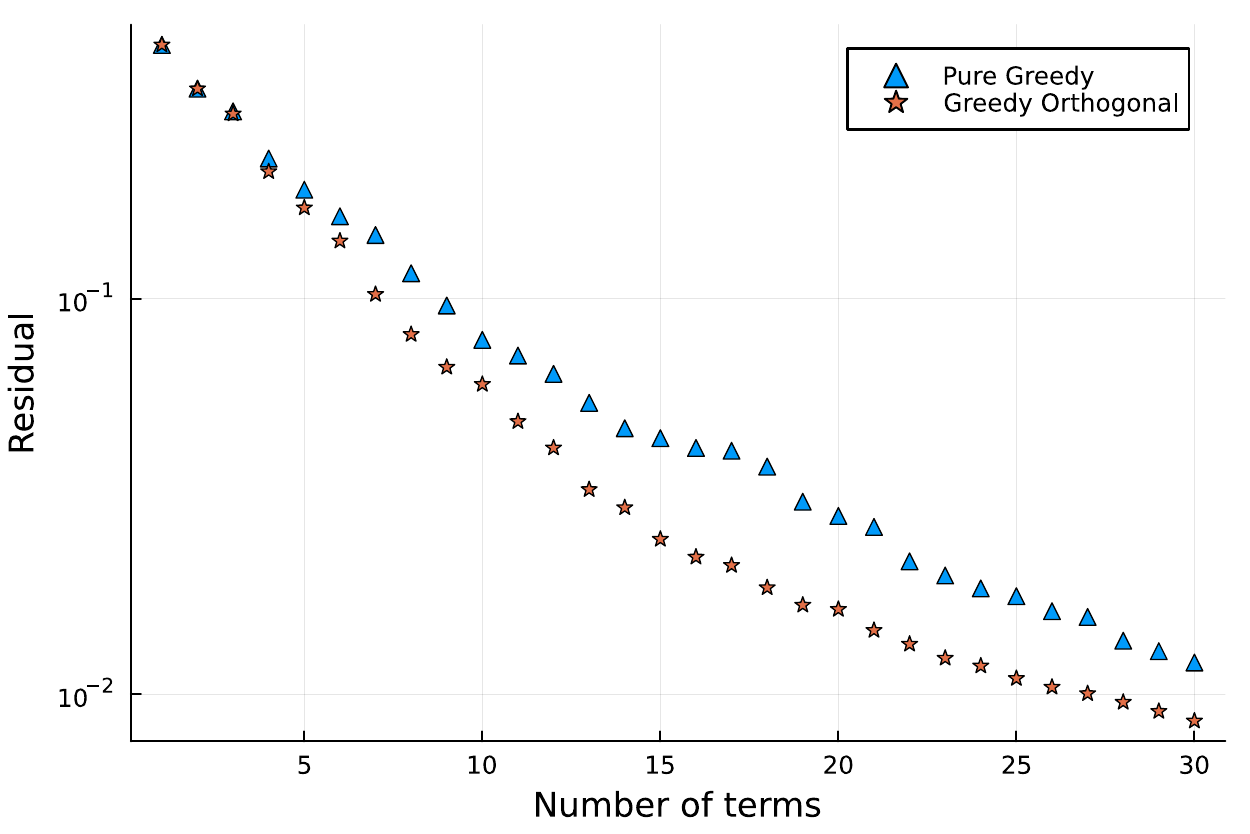}
		\caption{Decay of the residual with respect to the number of terms computed by the greedy algorithm.}
		\label{fig:residual evolution}
	\end{subfigure}
	\caption{} \label{fig:Gaussian error}
\end{figure}



\subsubsection{Scattering of a 3D wave packet} \label{subsec:Numerical 3D wave packet}

Although the one dimensional example studied in \Cref{subsec:Numerical 1D wave packet} is instructive, it offers very limited computational benefits, even when compared to a ``naive'' approach.
Indeed, in one dimension, considering the problem over a large space interval and relying on a classical discretization method (\textit{e.g.} finite elements or spectral methods) combined with a time steping scheme will yield a satisfying result very quickly.
The aim of this section is to demonstrate that the low-complexity approach can provide computational acceleration in three dimensions compared to classical methods.

Let us consider a 3D gaussian wave packet
\begin{equation}
    g_0(x) = \eexp^{-\frac 1 2 \abs{x - q}^2} \eexp^{ipx},
\end{equation}
with
\[
    q = \vecthree{3}{3}{0}, \quad
    p = \vecthree{\frac{-1}{\sqrt 2}}{\frac{-1}{\sqrt 2}}{0},
\]
and a repulsive potential barrier
\[
    V(x) = \eexp^{-\frac{\abs x^2}{2}}.
\]
We compute a reference solution using a classical Strang splitting scheme, given by the approximation $\eexp^{ih(\laplacian-V)} \approx \eexp^{i \frac h 2 \laplacian} \eexp^{-ihV} \eexp^{i \frac h 2 \laplacian}$.
The problem is discretized on a fixed domain $(-30,30)^3$ with Dirichlet discretization applied. A spectral basis is constructed using tensor products of rescaled $\sin(k\pi x)$ functions. We employ $128$ basis elements in each direction ($128^3$ elements in total).

On \Cref{fig:error evolution 3d}, we plot the evolution of the $\Ltwo$ norm of the error for a classical spectral method and the greedy method developped in this section, and on \Cref{fig:residual evolution3D} we plot the evolution of the residual $F$ with respect to the number of terms of the greedy algorithm. As in \Cref{subsec:Numerical 1D wave packet}, we observe a slow but steady decay for the greedy algorithm.

In \Cref{tab:computation times spectral} and \Cref{tab:computation times greedy}, we compare the computational costs of the methods. We observe that, even in three dimensions, the variational approach becomes competitive when compared to moderately fine discretizations. Moreover, the cost of the classical spectral methods rises significantly with finer discretization, while the cost of the greedy algorithm scales more moderately. This effect is expected to become even more pronounced in higher dimension as the cost of classical methods become prohibitive.
It is worth noting that both methods are eligible for parallelization, resulting in a much shorter real computation time.

\begin{figure}[h] \label{fig:Gaussian3D simulation}
	\centering
	\begin{subfigure}[t]{0.49\textwidth}
		\centering
		\includegraphics[width=\textwidth]{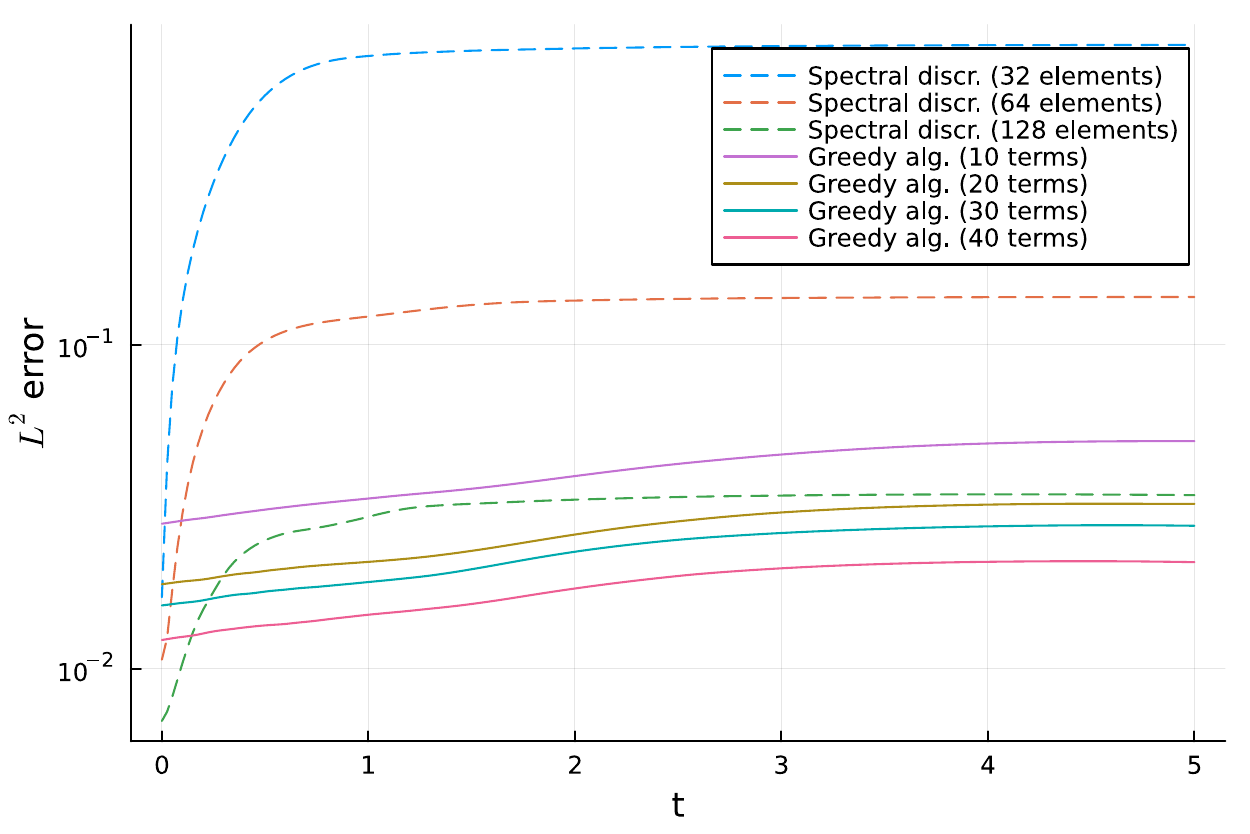}
		\caption{Evolution of the $\Ltwo$ error with respect to time for several methods and discretization. The error is obtained by comparing to a reference solution computed with a fine spectral discretization (128 per direction).}
		\label{fig:error evolution 3d}
	\end{subfigure}
	\hfill
	\begin{subfigure}[t]{0.49\textwidth}
		\centering
		\includegraphics[width=\textwidth]{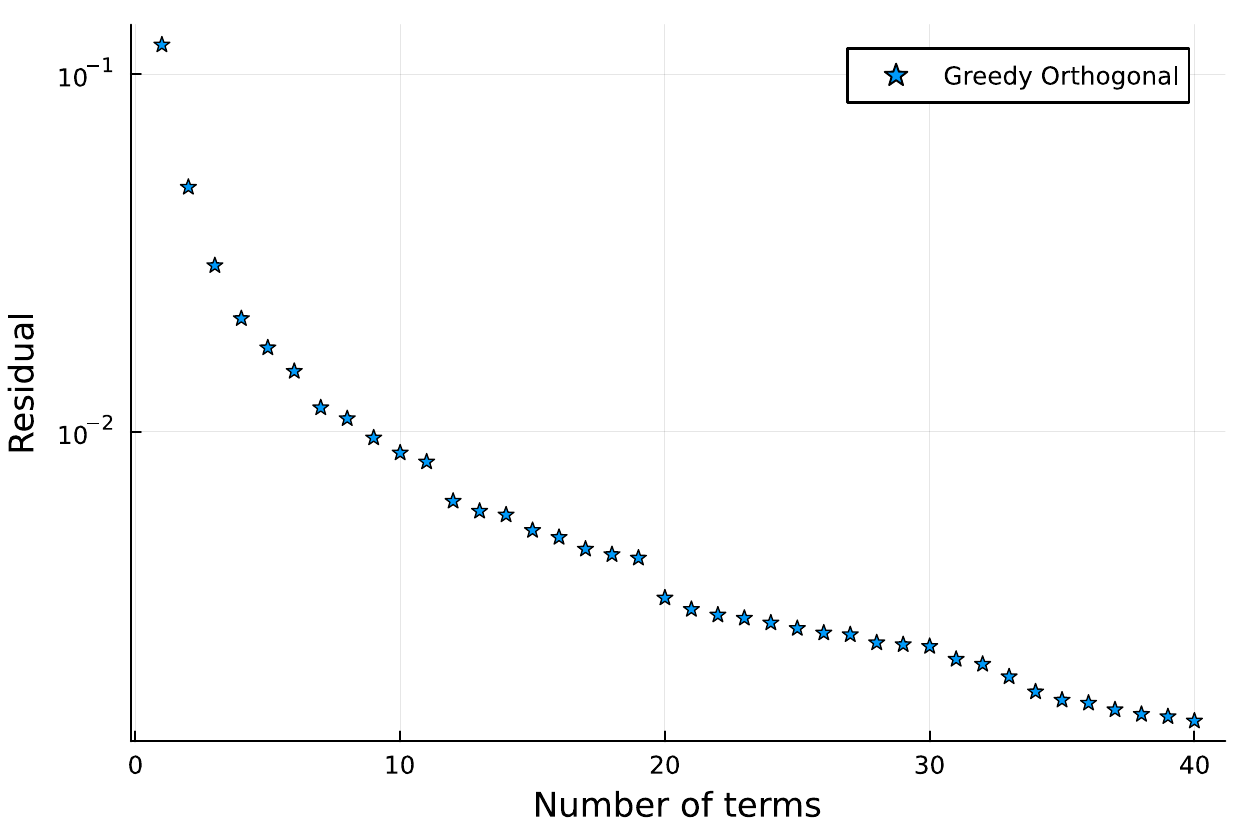}
		\caption{Decay of the residual with respect to the number of terms computed by the greedy algorithm.}
		\label{fig:residual evolution3D}
	\end{subfigure}
	\caption{} \label{fig:Gaussian3D error}
\end{figure}

\begin{table}[H]
  \centering
    \begin{tabular}{|l|c|c|c|c|}
        \hline
        \makecell{Number of points}
        & \makecell{32} 
        & \makecell{64} 
        & \makecell{80}
        & \makecell{128} \\
        \hline
        \makecell{Computation\\cost (cpu s)} 
        & 1
        & 14 
        & 25
        & 202 \\
        \hline
    \end{tabular}
    \caption{Computation times for the spectral method with respect to different number of points in each direction.}
    \label{tab:computation times spectral}

    \begin{tabular}{|l|c|c|c|c|}
        \hline
        \makecell{Number of terms}
        & \makecell{10} 
        & \makecell{20} 
        & \makecell{30}
        & \makecell{40} \\
        \hline
        \makecell{Computation\\cost (cpu s)} 
        & 18
        & 52
        & 106
        & 180 \\
        \hline
    \end{tabular}
    \caption{Computation times for the greedy method with respect to the number of terms.}
    \label{tab:computation times greedy}
\end{table}

\appendix

\section{Proof of \Cref{prop:Gaussian dictionary}} \label{app:Proof of weak closedness}

\begin{proposition} \label{prop:Gaussian density}
    Let $d \geq 1$, $\nu \in \N^d$, and $1 \leq r < \infty$.
    Then $\setspan \mathcal G_\nu$ is a dense subset of $\Lspacesp{r}{\R^d}$.
\end{proposition}
\begin{proof}
    Assume the conclusion does not hold. Then, there exists $u \in (\Lspacesp{r}{\R^d})' = \Lspacesp{r'}{\R^d}$ (where $r'$ is the conjugated exponent of $r$) such that $u \neq 0$ and $\int_{\R^d} \diff x u(x) g(x) = 0$ for any $g \in \mathcal G_\nu$.
    For $\varepsilon$, let $g_\varepsilon(x) = (\varepsilon \pi)^{- \frac d 2} \eexp^{- \frac{\abs x^2}{2\varepsilon}}$.
    Then, $g_\varepsilon \ast u = 0$ for any $\varepsilon > 0$.
    On the other hand, $g_\varepsilon \ast u \to u$ in $\Distribsetsp{\R^d}$ as $\varepsilon \to 0$, hence $u = 0$, which is a contradiction.
\end{proof}

\begin{proposition}\label{prop:General gaussian closedness}
    Let $d \geq 1$, $\nu \in \N^d$, and $1 \leq r < \infty$.
    Let $(g_k)_{k \in \N} \subset \mathcal G_\nu$.
    For all $k\in \mathbb{N}$, let $P_k \in \C^\nu[X]$, $p_k, q_k \in \R^d$, $A_k \in \mathcal{S}_d^{+,*}$, $B_k\in \mathcal S_d$ so that for all $x\in \mathbb{R}^d$, $g_k(x) = P_k(x) \eexp^{- \frac 1 2 (x - q_k) \cdot Q_k(x - q_k) } \eexp^{i p_k \cdot (x - q_k)}$  with $Q_k = A_k + i B_k$. Let us assume that:
    \begin{itemize}
        \item $\sup_{k \in \N} \norm{g_k}_{\Lspace r} < \infty$;
        \item there exists $g \in \Lone_{loc}(\R^d)$ such that $\displaystyle g_k \mathop{\rightharpoonup}^{\Distribset}_{k\to +\infty} g$.
    \end{itemize}
    Then, one of the following assertions must hold:
	\begin{enumerate}
		\item[(i)] $g = 0$;
		\item[(ii)] $g(x) = P(x) \eexp^{- \frac 1 2 (x-q)\cdot Q(x - q) } \eexp^{i p \cdot (x - q)}$ is an element of $\mathcal G_\nu \setminus \quickset 0$. Moreover, $\displaystyle a_k \mathop{\longrightarrow}_{k\to +\infty} a$, $\displaystyle p_k \mathop{\longrightarrow}_{k\to +\infty} p$, $\displaystyle q_k \mathop{\longrightarrow}_{k\to +\infty} q$, $\displaystyle Q_k \mathop{\longrightarrow}_{k\to +\infty} Q$, $\displaystyle P_k \mathop{\longrightarrow}_{k\to +\infty} P$, and $\displaystyle \norm{g_k - g}_{\Lspace s} \mathop{\longrightarrow}_{k\to +\infty} 0$ for any $1 \leq s \leq \infty$.
	\end{enumerate}
	As a consequence, $\mathcal G_\nu$ is a weakly closed set in $\Lspacesp{r}{\R^d}$. \\
\end{proposition}

\begin{remark}
    The assumption $g \in \Lone_{loc}(\R^d)$ can be omitted when $1 < r < \infty$, as it is automatically satisfied that $g \in \Lspacesp{r}{\R^d}$. However, it becomes essential in the case $r = 1$, as illustrated by the following example: Consider $g_k(x) = \sqrt{\frac{1}{2\pi k}} \eexp^{- \frac k 2 x^2}$, then $\norm{g_k}_{\Lonesp \R} = 1$ for any $k$, and $\displaystyle g_k \mathop{\rightharpoonup}^{\Distribset}_{k\to +\infty} \delta_0$.
\end{remark}

\medskip

We first establish a useful result on oscillatory integrals, which is a weaker version of Theorem XI.15 p.39 in \cite{reed_iii_1979}. For the sake of completeness, we include the proof which is much easier in our rather simple setup.
\begin{lemma} \label{lem:Quadratic phase estimate}
	Let $d\in \mathbb{N}^*$ and $A \in \matsetsquare \R d$ be a nonsingular symmetric matrix.
	Then, there exists a constant $C > 0$ such that
	\[
		\forall u \in \mathcal D(\R^d), \quad \forall \omega > 0, \quad
		\abs{\int_{\mathbb{R}^d}\diff x\eexp^{\frac {i\omega} 2 x\cdot Ax } u(x) }
		\leq	C \omega^{-d/2} \sup_{\alpha \in \mathbb{N}^d, \; \abs \alpha \leq d+1} \norm{\partial^\alpha u}_{\Lspace 1}.
	\]
\end{lemma}

\begin{proof}
	Let us denote here by $\mathcal F$ the Fourier transform and recall here the equality in the sense of tempered distributions for some $\omega>0$: for $\xi \in \mathbb{R}^d$,
	\[
		\Fourier\left(\eexp^{\frac {i\omega} 2 x\cdot Ax}\right)(\xi)
		=	\frac{(2\pi)^{d/2} \eexp^{\frac{i\pi}{4} \sgn A}}{\abs{\det(\omega A)}^{1/2}} \eexp^{- \frac i {2\omega} \xi \cdot A^{-1} \xi},
	\]
    where $\sgn A$ is the signature of $A$, \textit{i.e.} the difference between the number of positive and negative eigenvalues.
	Therefore, for all $u\in \mathcal D(\mathbb{R}^d)$, denoting by $C$ an arbitrary positive constant independent of $\omega$ and $u$,
	\[
	\begin{split}
		\abs{\int_{\mathbb{R}^d} \diff x \eexp^{\frac {i\omega} 2 x^TAx } u(x)}
		&=	(2\pi)^{-d} \abs{\int_{\mathbb{R}^d} \diff \xi \Fourier \left(\eexp^{\frac {i\omega} 2 x^TAx}\right)(\xi) \Fourier u(\xi)}
		\leq	C \omega^{-d/2} \int_{\mathbb{R}^d} \diff \xi \abs{\Fourier u(\xi)}	\\
		&\leq	C \omega^{-d/2} \int_{\mathbb{R}^d} \diff \xi (1 + \abs \xi)^{-(d+1)} (1 + \abs \xi)^{d+1} \abs{\Fourier u(\xi)}	\\
		&\leq	C  \omega^{-d/2} \norm{(1 + \abs \xi)^{d+1} \Fourier u}_{\Lspace \infty}
		\leq	C \omega^{-d/2} \sup_{\alpha \in \mathbb{N}^d, \; \abs \alpha \leq d+1} \norm{\partial^\alpha u}_{\Lspace 1}.
	\end{split}
	\]
\end{proof}

\begin{proof}[Proof of \Cref{prop:General gaussian closedness}]
    Throughout the proof, we use the same notation as in \Cref{prop:General gaussian closedness}.
    Let us first define the following quantities for any $k \in \N$:
    \begin{itemize}
        \item $\alpha_k = \norm{A_k}, \quad \text{and} \quad \beta_k = \norm{B_k} \quad \text{($\norm \cdot$ holds for the usual operator norm on $\R^{d \times d}$)};$
        \item \( \delta_k = \inf_{\substack{x \in \R^d \\ \abs x = 1}} x \cdot A_k x > 0 \); 
        \item For any $k \in \N$, $A_k = O_k \Lambda_k \transpose{O_k}$ with $O_k$ some real orthogonal matrix and $\Lambda_k = {\rm diag}(\lambda_k)$ where $\lambda_k = \lambda_{k, 1}, ..., \lambda_{k, d}$ with $0 < \lambda^k_1 \leq \lambda^k_2 \leq ... \leq \lambda^k_d$;
        \item For any $k \in \N$, $B_k = V_k M_k \transpose{V_k}$ with $V_k$ some real orthogonal matrix and $M_k = {\rm diag}(\mu_{k, 1}, ..., \mu_{k, d})$ and $\abs{\mu_{k, 1}} \leq ... \leq |\mu_{k, d}|$;
        \item For any $\theta \in \N^d, x \in \R^d$ and $\mu = (\mu_1, ..., \mu_d) \in (\R_+^*)^d$, let us define $x^\theta = x_1^{\theta_1} ... x_d^{\theta_d}$, and $h_\theta(\mu, x) = (\mu_1 ... \mu_d)^{\frac{1}{2r}} (\sqrt{\mu_1} x_1)^{\theta_1} ... (\sqrt{\mu_d} x_d)^{\theta_d}$.
        Then, there exists some $a_{k, \theta} \in \C$ such that
        \[
            P_k(x)
            =   \sum_{\theta \leq \nu} a_{k, \theta} h_\theta(\lambda_k, \transpose{O_k} (x - q_k)).
        \]
    \end{itemize}

    From now on, \underline{let us assume that $g$ is different from $0$}. We will prove successively that:
    \begin{enumerate}
        \item The coefficients $a_{k, \theta}$ are uniformly bounded with respect to \( k \); otherwise, the norm of \( (g_k) \) would become unbounded.
        \item There exists a \( \delta > 0 \) such that \( \delta_k \geq \delta \) for any \( k \in \mathbb{N} \); otherwise, the sequence \( (g_k) \) would spread, causing the limit \( g \) to be zero.
        \item The sequence \( q_k \) is uniformly bounded with respect to \( k \); otherwise, the mass of \( g_k \) would escape to infinity.
        \item The sequence \( \alpha_k \) is uniformly bounded with respect to \( k \) to prevent mass concentration on a set of measure zero.
        \item The sequence \( \beta_k \) is uniformly bounded with respect to \( k \); otherwise, oscillations would cause the limit to vanish.
        \item The sequence \( p_k \) is uniformly bounded with respect to \( k \) for the same reason as above.
    \end{enumerate}

    {\bfseries Step~1: The $a_{k, \theta}$ are bounded uniformly}
    
    \medskip
    For any $k \in \N$, we obtain from the change of variable $x = q_k + O_k \Lambda_k^{- \frac 1 2} y$
    \[
    \begin{split}
        \norm{g_k}_{\Lspace r}
        &=  \left( \int \diff x \abs{\sum_{\theta \leq \nu} a_{k, \theta} h_\theta(\lambda_k, \transpose{O_k} (x - q_k)) \eexp^{- \frac 1 2 A_k(x - q_k) \cdot (x - q_k)}}^r \right)^{\frac 1 r}    \\
        &=  \left( \int \diff y \abs{\sum_{\theta \leq \nu} a_{k, \theta} y^\theta \eexp^{- \frac 1 2 \abs y^2}}^r \right)^{\frac 1 r} \\
        &\geq   \frac 1 C \sup_{\theta \leq \nu} \abs{a_{k, \theta}}.
    \end{split}
    \]
    For the last inequality, we used the fact that the $(y \mapsto y^\theta)_{\theta \leq \nu}$ are linearly independent, and that all norms are equivalent in finite dimension.
    Since, by assumption, the $(g_k)$ are uniformly bounded in $\Lspace r$, we obtain the desired uniform boundedness for the $a_{k, \theta}$.

    \medskip
    {\bfseries Step~2: There is a $\delta > 0$ such that $\delta_k \geq \delta$ for all $k$}
    
    \medskip
    Let us assume the assertion does not hold. Then, up to a subsequence, we may assume that $\lim_{k \to \infty} \delta_k = \lim_{k \to \infty} \lambda_{k, 1} = 0$. \\
    Let $\varphi \in \TestFunctionsetsp{\R^d}$, and $R > 0$ be such that $\supp \varphi \subset \openball 0 R$.
    Then, the change of variable $x = q_k + O_k \Lambda_k^{-\frac 1 2} y$ yields:
	\[
	\begin{split}
		\abs{\dualitybracket{g_k}{\varphi}_{\Distribset, \TestFunctionset}}
        \leq    \norm{\varphi}_{\Lspace{r'}} \norm{g_k}_{\Lspacesp{r}{\openball 0 R}}
		&\leq	\norm{\varphi}_{\Lspace{r'}}
		\left( \int_{\Lambda_k^{\frac 1 2} \transpose{O_k} \openball{-q_k}{R}} \diff y \abs{\sum_{\theta \leq \nu} a_{k, \theta} y^\theta \eexp^{- \frac 1 2 \abs y^2}}^r \right)^{\frac 1 r}	\\
		&\leq	C \norm{\varphi}_{\Lspace{r'}} \sum_{\theta \leq \nu}
		\left( \int_{\Lambda_k^{\frac 1 2} \openball{-  \transpose{O_k} q_k}{R}} \diff y \abs{y^\theta \eexp^{- \frac 1 2 \abs y^2}}^r \right)^{\frac 1 r}.
	\end{split}
	\]
	The functions $\left(y^\theta \eexp^{- \frac 1 2 \abs y^2}\right)^r$ are integrable on $\R^d$. 
	If $\lim_{k \to \infty} \lambda_{k, 1} = 0$, then we have that 
	\[
		\lim_{k \to \infty} \int_{\Lambda_k^{\frac 1 2} \openball{-  \transpose{O_k} q_k}{R}} \diff y \abs{y^\theta \eexp^{- \frac 1 2 \abs y^2}}^r =0,
	\]
    since the integration domain is shrinking along the first dimension. \\
    This implies that $g = 0$ in $\Distribsetsp{\R^d}$, which is a contradiction.

    \medskip
    {\bfseries Step~3: The $q_k$ are bounded uniformly with respect to $k$}
    
    \medskip
    If the assertion does not hold, up to a subsequence, we may assume that $\lim_{k \to \infty} \abs{q_k} = \infty$. 
    Since by Step 2, the $(g_k)$ do not spread, then center of mass of the $g_k$ goes infinitely away. 
    This means that $g=0$ which is again a contradiction.

    \medskip
    {\bfseries Step~4: The $\gamma_k$ are bounded uniformly with respect to $k$}
    
    \medskip
    Once again, we assume by contradiction and up to an extraction that $\lim_{k \to \infty} \gamma_k = \lim_{k \to \infty} \lambda_{k, d} = \infty$. This time, up to another extraction, we also assume that $q_k$ and $O_k$ converge respectively to some $q_\infty$ and $O_\infty$. By defining $\tilde g_k(y) = g_k(q_k + O_k y)$, it can actually be proved that $\displaystyle \tilde g_k \mathop{\rightharpoonup}_{k\to +\infty}^{\Distribset} \tilde g$ with $\tilde g(y) = g(q_\infty + O_\infty y)$. This is a consequence of the fact that $\sup_k \norm{\tilde g_k}_{\Lspace r} = \sup_k \norm{g_k}_{\Lspace r} < \infty$, and that for any $\varphi \in \TestFunctionsetsp{\R^d}$ one has $\lim_{k \to \infty} \norm{\varphi(\transpose{O_k} (\cdot - q_k)) - \varphi(\transpose{O_\infty}(\cdot - q_\infty))}_{\Lspace{r'}} = 0$. \\
    Let $\varphi \in \TestFunctionsetsp{\R^d}$ be such that $\supp \varphi \subset \quickset{\abs{x_d} \geq \varepsilon}$ for some $\varepsilon > 0$. Then, following similar computation as previously, we obtain that
    \[
        \abs{\dualitybracket{\tilde g_k}{\varphi}_{\Distribset, \TestFunctionset}}
        \leq    C \norm{\varphi}_{\Lspace{r'}} \sum_{\theta_d = 0}^{\nu_d} \left( \int_{\quickset{\abs t \geq \lambda_{k, d} \varepsilon}} \diff t \abs{t}^{\theta_d r} \eexp^{-\frac r 2 t^2}  \right)^{\frac 1 r}
        \xrightarrow[k \to \infty]{} 0.
    \]
    We conclude that $\tilde g$ is an $\Lone_{loc}$ function equal to zero almost everywhere outside of the hyperplane $\quickset{x_d = 0}$, hence $\tilde g = 0$, and consequently $g = 0$, which is again a contradiction.

    \medskip
    {\bfseries Step~5: The $\beta_k$ are bounded uniformly with respect to $k$}
    
    \medskip
    By contradiction, assume $\lim_{k \to \infty} \beta_k = \lim_{k \to \infty} \abs{\mu_{k, d}} = \infty$. \\
    Let $\tilde O_k = \transpose{O_k} V_k$, $\tilde A_k = \transpose{V_k} A_k V_k$, and $\tilde p_k = \transpose{V_k} p_k = (b_{k, 1}, ..., b_{k, d})$.
    Then the change of variable $x = q_k + V_k y$ yields
	\[
	\begin{split}
		\abs{\dualitybracket{g_k}{\varphi}_{\Distribset, \TestFunctionset}}
		&=	\abs{\int_{\R^d} \diff x \sum_{\theta \leq \nu} \varphi(x) a_{k, \theta} h_\theta(\lambda_k, \transpose{O_k}(x - q_k)) \eexp^{- \frac 1 2(x - q_k) \cdot Q_k(x - q_k)} \eexp^{i p_k \cdot (x - q_k)}}	\\
		&=	\abs{\int_\R \diff{y_d} \eexp^{- \frac i 2 \mu_{k, d} (y_d - \mu_{k, d}^{-1} b_{k, d})^2} I_k(y_d)}, \\
        &=  \abs{\int_\R \diff t \eexp^{- \frac i 2 \mu_{k, d} t^2} I_k(\mu_{k, d}^{-1} b_{k, d} + t)},
	\end{split}
	\]
    where
    \[
        I_k(y_d)
        =   \sum_{\theta \leq \nu} a_{k, \theta} \int_{\R^{d-1}} \diff{y_1} ... \diff{y_{d-1}} \varphi(q_k + V_k y) h_\theta(\lambda_k, \tilde O_k y) \eexp^{- \frac 1 2 \tilde A_k y \cdot y} \eexp^{- \frac i 2 \sum_{j=1}^{d-1} \mu_{k, j} y_j^2}  \eexp^{i \sum_{j=1}^{d-1} b_{k, j} y_j}.
    \]
    We can then apply \Cref{lem:Quadratic phase estimate} with $A = 1 \in \R^{1 \times 1}$ to obtain
    \[
        \abs{\dualitybracket{g_k}{\varphi}_{\Distribset, \TestFunctionset}}
        \leq    C \beta_k^{- \frac 1 2} (\norm{I_k}_{\Lonesp{\R}} + \norm{I_k'}_{\Lonesp{\R}} + \norm{I_k''}_{\Lonesp{\R}}).
    \]
    It follows from the previous steps of the proof that $\norm{I_k}_{\Lonesp{\R}} + \norm{I_k'}_{\Lonesp{\R}} + \norm{I_k''}_{\Lonesp{\R}}$ is bounded uniformly with respect to $k$.

    
    Then we have
    \[
        \abs{\dualitybracket{g_k}{\varphi}_{\Distribset, \TestFunctionset}}
        \leq    C \beta_k^{- \frac 1 2}
        \xrightarrow[k \to \infty]{} 0.
    \]

    \medskip
    {\bfseries Step~6: The $p_k$ are bounded uniformly with respect to $k$}
    
    \medskip
    By contradiction, assume that $\lim_{k \to \infty} \abs{p_k} = \infty$. \\
    Let $\varphi \in \TestFunctionsetsp{\R^d}$, we compute
    \[
    \begin{split}
        \dualitybracket{g_k}{\varphi}_{\Distribset, \TestFunctionset}
        =   \int_{\R^d} \diff y \eta_k(y) \eexp^{ip_k \cdot y},
        \quad \text{with} \;
        \eta_k(y)
        =   \sum_{\theta \leq \nu}   \varphi(y + q_k) h_\theta(\lambda_k, \transpose{O_k}y) \eexp^{-\frac{1}{2} Q_k y \cdot y}.
    \end{split}
    \]
    It then follows from an integration by part that
    \[
        \abs{\dualitybracket{\tilde g_k}{\varphi}_{\Distribset, \TestFunctionset}}
        =   \abs{\int_{\R^d} \diff y \eexp^{ip_k \cdot y} \frac{p_k \cdot \nabla \eta_k(y)}{\abs{p_k}^2}}
        \leq    \frac{C}{\abs{p_k}}
        \xrightarrow[k \to \infty]{} 0.
    \]

    \medskip
    {\bfseries Final step:}
    
    \medskip
    We have, for any $k \in \N$,
    \[
        g_k(x)
        =   P_k(x) \eexp^{- \frac 1 2 Q_k(x - q_k) \cdot (x - q_k)} \eexp^{ip_k \cdot (x - q_k)},
    \]
    with
    \[
        P_k(x)
        =   \sum_{\theta \leq \nu} a_{k, \theta} h_\theta(\lambda_k, \transpose{O_k}(x - q_k)).
    \]
    It follows from the previous steps that $Q_k, q_k, p_k$ are uniformly bounded, as well as all the coefficients of the polynomial $P_k$.
    Moreover, they can only have one subsequential limit, that we denote by $Q, q, p, P$, since these limits are characterized by
    \[
        g(x)
        =   P(x) \eexp^{- \frac 1 2 Q(x - q) \cdot (x - q)} \eexp^{ip \cdot (x - q)}.
    \]
    (and this is indeed an element of $\Lspacesp{r}{\R^d}$, because $\Re(Q) \geq \delta > 0$) \\
    Therefore, the four sequences converge to these subsequential limits.
\end{proof}

\printbibliography

\end{document}